%% file: jour-14-fxr-main.tex
\documentclass[journal]{IEEEtran}

\usepackage{amsmath,amssymb,amsthm}
\usepackage{graphicx}
\usepackage{epstopdf}
\usepackage{subfigure}
\usepackage{caption}
\usepackage{setspace}
\usepackage{textcomp}
\usepackage{algorithmic}
\usepackage{algorithm}
\usepackage{cases}
\usepackage{bbm}
\usepackage{upgreek}
\usepackage[square,comma,numbers,sort&compress]{natbib}
\usepackage{hyperref}
\usepackage[all]{hypcap}

\hypersetup{colorlinks=true,linkcolor=blue,citecolor=red}

\input{def}

\newtheorem{theorem}{Theorem}

\begin{document}

\title{Fast X-Ray CT Image Reconstruction Using the Linearized Augmented Lagrangian Method with Ordered Subsets}
\author{Hung Nien, \textit{Student Member, IEEE}, and Jeffrey A. Fessler, \textit{Fellow, IEEE}\thanks{This work is supported in part by NIH grant R01-HL-098686 and by an equipment donation from Intel Corporation.}\\[5pt]Department of Electrical Engineering and Computer Science\\University of Michigan, Ann Arbor, MI}

\maketitle

\begin{abstract}
The augmented Lagrangian (AL) method that solves convex optimization problems with linear constraints \cite{hestenes:69:mag,powell:69:amf,glowinski:75:slp,gabay:76:ada,eckstein:92:otd} has drawn more attention recently in imaging applications due to its decomposable structure for composite cost functions and empirical fast convergence rate under weak conditions. However, for problems such as X-ray computed tomography (CT) image reconstruction and large-scale sparse regression with ``big data'', where there is no efficient way to solve the inner least-squares problem, the AL method can be slow due to the inevitable iterative inner updates \cite{ramani:12:asb,mcgaffin:12:rma}. In this paper, we focus on solving regularized (weighted) least-squares problems using a linearized variant of the AL method \cite{lin:11:lad,wang:12:tla,yang:13:lal,xiao:13:sal} that replaces the quadratic AL penalty term in the scaled augmented Lagrangian with its separable quadratic surrogate (SQS) function \cite{erdogan:99:osa,kim:12:paf}, thus leading to a much simpler ordered-subsets (OS) \cite{erdogan:99:osa} accelerable splitting-based algorithm, \mbox{OS-LALM}, for X-ray CT image reconstruction. To further accelerate the proposed algorithm, we use a second-order recursive system analysis to design a deterministic downward continuation approach that avoids tedious parameter tuning and provides fast convergence. Experimental results show that the proposed algorithm significantly accelerates the ``convergence'' of X-ray CT image reconstruction with negligible overhead and greatly reduces the OS artifacts \cite{kim:13:osw,kim:13:axr,kim:13:osa} in the reconstructed image when using many subsets for OS acceleration.
\end{abstract}

\section{Introduction} \label{sec:jour-14-fxr:intro}
\IEEEPARstart{S}{tatistical} methods for image reconstruction has been explored extensively for computed tomography (CT) due to the potential of acquiring a CT scan with lower X-ray dose while maintaining image quality. However, the much longer computation time of statistical methods still restrains their applicability in practice. To accelerate statistical methods, many optimization techniques have been investigated. The augmented Lagrangian (AL) method (including its alternating direction variants) \cite{hestenes:69:mag,powell:69:amf,glowinski:75:slp,gabay:76:ada,eckstein:92:otd} is a powerful technique for solving regularized inverse problems using variable splitting. For example, in total-variation (TV) denoising and compressed sensing (CS) problems, the AL method can separate non-smooth $\ell_1$ regularization terms by introducing auxiliary variables, yielding simple penalized least-squares inner problems that are solved efficiently using the fast Fourier transform (FFT) algorithm and proximal mappings such as the soft-thresholding for the $\ell_1$ norm \cite{goldstein:09:tsb,afonso:11:aal}. However, in applications like X-ray CT image reconstruction, the inner least-squares problem is challenging due to the highly shift-variant Hessian caused by the huge dynamic range of the statistical weighting. To solve this problem, Ramani \textit{et al.} \cite{ramani:12:asb} introduced an additional variable that separates the shift-variant and approximately shift-invariant components of the statistically weighted quadratic data-fitting term, leading to a better-conditioned inner least-squares problem that was solved efficiently using the preconditioned conjugate gradient (PCG) method with an appropriate circulant preconditioner. Experimental results showed significant acceleration in $2$D CT \cite{ramani:12:asb}; however, in $3$D CT, due to different cone-beam geometries and scan trajectories, it is more difficult to construct a good preconditioner for the inner least-squares problem, and the method in \cite{ramani:12:asb} has yet to achieve the same acceleration as in $2$D CT. Furthermore, even when a good preconditioner can be found, the iterative PCG solver requires several forward/back-projection operations per outer iteration, which is very time-consuming in $3$D CT, significantly reducing the number of outer-loop image updates one can perform within a given reconstruction time.

The ordered-subsets (OS) algorithm \cite{erdogan:99:osa} is a first-order method with a diagonal preconditioner that uses somewhat conservative step sizes but is easily applicable to $3$D CT. By grouping the projections into $M$ ordered subsets that satisfy the ``subset balance condition'' and updating the image incrementally using the $M$ subset gradients, the OS algorithm effectively performs $M$ times as many image updates per outer iteration as the standard gradient descent method, leading to $M$ times acceleration in early iterations. We can interpret the OS algorithm and its variants as incremental gradient methods \cite{bertsekas:10:igs}; when the subset is chosen randomly with some constraints so that the subset gradient is unbiased and with finite variance, they can also be referred as stochastic gradient methods \cite{robbins:51:asa} in the machine learning literature. Recently, OS variants \cite{kim:13:osw,kim:13:axr} of the fast gradient method \cite{nesterov:83:amf,nesterov:05:smo,beck:09:afi} were also proposed and demonstrated dramatic acceleration (about $M^2$ times in early iterations) in convergence rate over their one-subset counterparts. However, experimental results showed that when $M$ increases, fast OS algorithms seem to have ``larger'' limit cycles and exhibit noise-like OS artifacts in the reconstructed images \cite{kim:13:osa}. This problem is also studied in the machine learning literature. Devolder showed that the error accumulation in fast gradient methods is inevitable when an inexact oracle is used, but it can be reduced by using relaxed momentum, i.e., a growing diagonal majorizer (or equivalently, a diminishing step size), at the cost of slower convergence rate \cite{devolder:11:sfo}. Schmidt \textit{et al.} also showed that an accelerated proximal gradient method is more sensitive to errors in the gradient and proximal mapping calculation of the smooth and non-smooth cost function components, respectively \cite{schmidt:11:cro}.

OS-based algorithms, including the standard one and its fast variants, are not convergent in general (unless relaxation \cite{ahn:03:gci} or incremental majorization \cite{ahn:06:cio} is used, unsurprisingly, at the cost of slower convergence rate) and possibly introduce noise-like artifacts; however, the effective $M$-times image updates using OS is still very promising for AL methods. Recently, Ouyang \textit{et al.} \cite{ouyang:13:sad} proposed a stochastic setting for the alternating direction method of multipliers (ADMM) \cite{eckstein:92:otd,afonso:11:aal} that majorizes smooth data-fitting terms such as the logistic loss in the scaled augmented Lagrangian using a growing diagonal majorizer with stochastic gradients. Unlike methods in \cite{ramani:12:asb,mcgaffin:12:rma}, which introduced an additional variable for better-conditioned inner least-squares problem, the diagonal majorization in stochastic ADMM eliminates the difficult least-squares problem involving the system matrix (e.g., the forward projection matrix in CT) that must be solved in standard ADMM. In fact, only part of the data has to be visited (for evaluating the gradient of a subset of the data) per stochastic ADMM iteration. Therefore, the cost per stochastic ADMM iteration is reduced substantially, and one can run more stochastic ADMM iterations for better reconstruction in a given reconstruction time. However, the growing diagonal majorizer (used to ensure convergence) inevitably slows the convergence of stochastic ADMM from $\fx{O}{1/k}$ to $\sfx{O}{1/\ts\sqrt{Mk}}$, where $k$ denotes the number of effective passes of the data, and $M$ denotes the number of subsets. To achieve significant acceleration, $M$ should be much greater than $k$, and this will increase both the variance of the subset gradients and the cost per effective pass of the data. Therefore, in X-ray CT image reconstruction, stochastic ADMM (with a growing diagonal majorizer) is not efficient.

In this paper, we focus on solving regularized (weighted) least-squares problems using a linearized variant of the AL method \cite{lin:11:lad,wang:12:tla,yang:13:lal,xiao:13:sal}. We majorize the quadratic AL penalty term, instead of the smooth data-fitting term, in the scaled augmented Lagrangian using a fixed diagonal majorizer, thus leading to a much simpler OS-accelerable splitting-based algorithm, \mbox{OS-LALM}, for X-ray CT image reconstruction. To further accelerate the proposed algorithm, we use a second-order recursive system analysis to design a deterministic downward continuation approach that avoids tedious parameter tuning and provides fast convergence. Experimental results show that the proposed algorithm significantly accelerates the ``convergence'' of X-ray CT image reconstruction with negligible overhead and greatly reduces the OS artifacts in the reconstructed image when using many subsets.

The paper is organized as follows. Section~\ref{sec:jour-14-fxr:background} reviews the linearized AL method in a general setting and shows new convergence properties of the linearized AL method with inexact updates. Section~\ref{sec:jour-14-fxr:proposed_alg} derives the proposed OS-accelerable splitting-based algorithm for solving regularized least-squares problems using the linearized AL method and develops a deterministic downward continuation approach for fast convergence without parameter tuning. Section~\ref{sec:jour-14-fxr:impl_detail} considers solving X-ray CT image reconstruction problem with penalized weighted least-squares (PWLS) criterion using the proposed algorithm. Section~\ref{sec:jour-14-fxr:result} reports the experimental results of applying our proposed algorithm to X-ray CT image reconstruction. Finally, we draw conclusions in Section~\ref{sec:jour-14-fxr:conclusion}.

\section{Background} \label{sec:jour-14-fxr:background}
\subsection{Linearized AL method} \label{subsec:jour-14-fxr:lalm}
Consider a general composite convex optimization problem:
\begin{equation} \label{eq:jour-14-fxr:comp_conv_opt_prob}
	\hat{\mb{x}}
	\in
	\bargmin{\mb{x}}{\fx{g}{\mb{Ax}}+\fx{h}{\mb{x}}}
\end{equation}
and its equivalent constrained minimization problem:
\begin{equation} \label{eq:jour-14-fxr:eq_comp_conv_opt_prob}
	\left(\hat{\mb{x}},\hat{\mb{u}}\right)
	\in
	\bargmin{\mb{x},\mb{u}}{\fx{g}{\mb{u}}+\fx{h}{\mb{x}}}
	\text{ s.t. }
	\mb{u}=\mb{Ax} \, ,
\end{equation}
where both $g$ and $h$ are closed and proper convex functions. Typically, $g$ is a weighted quadratic data-fitting term, and $h$ is an edge-preserving regularization term in CT. One way to solve the constrained minimization problem \eqref{eq:jour-14-fxr:eq_comp_conv_opt_prob} is to use the (alternating direction) AL method, which alternatingly minimizes the scaled augmented Lagrangian:
\begin{equation} \label{eq:jour-14-fxr:aug_lagrangian}
	\fx{\LA}{\mb{x},\mb{u},\mb{d};\rho}
	\teq
	\fx{g}{\mb{u}}+\fx{h}{\mb{x}}+\ts\frac{\rho}{2}\norm{\mb{Ax}-\mb{u}-\mb{d}}{2}^2
\end{equation}
with respect to $\mb{x}$ and $\mb{u}$, followed by a gradient ascent of $\mb{d}$, yielding the following AL iterates \cite{eckstein:92:otd,afonso:11:aal}:
\begin{equation} \label{eq:jour-14-fxr:general_al_iterates}
	\begin{cases}
	\iter{\mb{x}}{k+1}
	\in
	\argmin
	{\mb{x}}
	{\fx{h}{\mb{x}}+\ts\frac{\rho}{2}\norm{\mb{Ax}-\iter{\mb{u}}{k}-\iter{\mb{d}}{k}}{2}^2} \\
	\iter{\mb{u}}{k+1}
	\in
	\argmin
	{\mb{u}}
	{\fx{g}{\mb{u}}+\ts\frac{\rho}{2}\norm{\mb{A}\iter{\mb{x}}{k+1}-\mb{u}-\iter{\mb{d}}{k}}{2}^2} \\
	\iter{\mb{d}}{k+1}
	=
	\iter{\mb{d}}{k}-\mb{A}\iter{\mb{x}}{k+1}+\iter{\mb{u}}{k+1} \, ,
	\end{cases}
\end{equation}
where $\mb{d}$ is the scaled Lagrange multiplier of the split variable $\mb{u}$, and $\rho>0$ is the corresponding AL penalty parameter.

In the linearized AL method \cite{lin:11:lad,wang:12:tla,yang:13:lal,xiao:13:sal} (also known as the split inexact Uzawa method \cite{esser:10:agf,zhang:10:bnr,zhang:11:aup}), one replaces the quadratic AL penalty term in the $\mb{x}$-update of \eqref{eq:jour-14-fxr:general_al_iterates}:
\begin{equation} \label{eq:jour-14-fxr:quad_al_penalty}
	\fx{\theta_k}{\mb{x}}\teq\ts\frac{\rho}{2}\norm{\mb{Ax}-\iter{\mb{u}}{k}-\iter{\mb{d}}{k}}{2}^2
\end{equation}
by its separable quadratic surrogate (SQS) function:
\begin{align} \label{eq:jour-14-fxr:sqs_quad_al_penalty}
	&\,\,\,\,\,\,\,\,
	\bfx{\breve{\theta}_k}{\mb{x};\iter{\mb{x}}{k}} \nonumber \\
	&\teq
	\bfx{\theta_k}{\iter{\mb{x}}{k}}
	+
	\biprod{\nabla\bfx{\theta_k}{\iter{\mb{x}}{k}}}{\mb{x}-\iter{\mb{x}}{k}}
	+
	\ts\frac{\rho L}{2}\norm{\mb{x}-\iter{\mb{x}}{k}}{2}^2 \nonumber \\
	&=
	\ts\frac{\rho}{2t}\norm{\mb{x}-\big(\iter{\mb{x}}{k}-t\mb{A}'\big(\mb{A}\iter{\mb{x}}{k}-\iter{\mb{u}}{k}-\iter{\mb{d}}{k}\big)\big)}{2}^2 \nonumber \\
	&\qquad\qquad\qquad\qquad\quad\,\,
	+(\text{constant independent of $\mb{x}$}) \, .
\end{align}
This function satisfies the ``majorization'' condition:
\begin{equation} \label{eq:jour-14-fxr:maj_cond}
	\begin{cases}
	\bfx{\breve{\theta}_k}{\mb{x};\bar{\mb{x}}}\geq\bfx{\theta_k}{\mb{x}} & ,\forall\mb{x},\bar{\mb{x}}\in\text{Dom}\,\theta_k \\
	\bfx{\breve{\theta}_k}{\bar{\mb{x}};\bar{\mb{x}}}=\bfx{\theta_k}{\bar{\mb{x}}}  & ,\forall\bar{\mb{x}}\in\text{Dom}\,\theta_k \, ,
	\end{cases}
\end{equation}
where $L>\norm{\mb{A}}{2}^2=\fx{\lambda_{\text{max}}}{\mb{A}'\mb{A}}$ ensures that $L\mb{I}\succ\mb{A}'\mb{A}$, and $t\teq1/L$. It is trivial to generalize $L$ to a symmetric positive semi-definite matrix $\mb{L}$, e.g., the diagonal matrix used in OS-based algorithms \cite{erdogan:99:osa,kim:13:aos}, and still ensure \eqref{eq:jour-14-fxr:maj_cond}. When $\mb{L}=\mb{A}'\mb{A}$, the linearized AL method reverts to the standard AL method. Majorizing with a diagonal matrix removes the entanglement of $\mb{x}$ introduced by the system matrix $\mb{A}$ and leads to a simpler $\mb{x}$-update. The corresponding linearized AL iterates are as follows \cite{lin:11:lad,wang:12:tla,yang:13:lal,xiao:13:sal}:
\begin{equation} \label{eq:jour-14-fxr:general_lal_iterates}
	\begin{cases}
	\iter{\mb{x}}{k+1}
	\in
	\argmin
	{\mb{x}}
	{\fx{\phi_k}{\mb{x}}\teq\fx{h}{\mb{x}}+\bfx{\breve{\theta}_k}{\mb{x};\iter{\mb{x}}{k}}} \\
	\iter{\mb{u}}{k+1}
	\in
	\argmin
	{\mb{u}}
	{\fx{g}{\mb{u}}+\ts\frac{\rho}{2}\norm{\mb{A}\iter{\mb{x}}{k+1}-\mb{u}-\iter{\mb{d}}{k}}{2}^2} \\
	\iter{\mb{d}}{k+1}
	=
	\iter{\mb{d}}{k}-\mb{A}\iter{\mb{x}}{k+1}+\iter{\mb{u}}{k+1} \, .
	\end{cases}
\end{equation}
The $\mb{x}$-update can be written as the proximal mapping of $h$:
\begin{align} \label{eq:jour-14-fxr:general_lal_iterates_x}
	\iter{\mb{x}}{k+1}
	&\in
	\bprox{\left(\rho^{-1}t\right)h}{\iter{\mb{x}}{k}-t\mb{A}'\big(\mb{A}\iter{\mb{x}}{k}-\iter{\mb{u}}{k}-\iter{\mb{d}}{k}\big)} \nonumber \\
	&=
	\bprox{\left(\rho^{-1}t\right)h}{\iter{\mb{x}}{k}-(\rho^{-1}t)\,\iter{\mb{s}}{k+1}}
	\, ,
\end{align}
where $\text{prox}_f$ denotes the proximal mapping of $f$ defined as:
\begin{equation} \label{eq:jour-14-fxr:def_prox}
	\prox{f}{\mb{z}}
	\teq
	\argmin{\mb{x}}{\fx{f}{\mb{x}}+\ts\frac{1}{2}\norm{\mb{x}-\mb{z}}{2}^2} \, ,
\end{equation}
and
\begin{equation} \label{eq:jour-14-fxr:def_search_dir}
	\iter{\mb{s}}{k+1}
	\teq
	\rho\mb{A}'\big(\mb{A}\iter{\mb{x}}{k}-\iter{\mb{u}}{k}-\iter{\mb{d}}{k}\big)
\end{equation}
denotes the ``search direction'' of the proximal gradient $\mb{x}$-update in \eqref{eq:jour-14-fxr:general_lal_iterates_x}. Furthermore, $\breve{\theta}_k$ can also be written as:
\begin{equation} \label{eq:jour-14-fxr:sqs_al_penalty_proximal_form}
	\bfx{\breve{\theta}_k}{\mb{x};\iter{\mb{x}}{k}}
	=
	\fx{\theta_k}{\mb{x}}+\ts\frac{\rho}{2}\norm{\mb{x}-\iter{\mb{x}}{k}}{\mb{G}}^2 \, ,
\end{equation}
where $\mb{G}\teq L\mb{I}-\mb{A}'\mb{A}\succ 0$ by the definition of $L$. Hence, the linearized AL iterates \eqref{eq:jour-14-fxr:general_lal_iterates} can be represented as a proximal-point variant of the standard AL iterates \eqref{eq:jour-14-fxr:general_al_iterates} (also known as the preconditioned AL iterates) by plugging \eqref{eq:jour-14-fxr:sqs_al_penalty_proximal_form} into \eqref{eq:jour-14-fxr:general_lal_iterates} \cite{esser:10:agf,chambolle:11:afo,ouyang:13:sad}:
\begin{equation} \label{eq:jour-14-fxr:general_lal_iterates_as_prox_al}
	\begin{cases}
	\iter{\mb{x}}{k+1}
	\in
	\argmin
	{\mb{x}}
	{
	\fx{h}{\mb{x}}+\fx{\theta_k}{\mb{x}}+\ts\frac{\rho}{2}\norm{\mb{x}-\iter{\mb{x}}{k}}{\mb{G}}^2
	} \\
	\iter{\mb{u}}{k+1}
	\in
	\argmin
	{\mb{u}}
	{\fx{g}{\mb{u}}+\ts\frac{\rho}{2}\norm{\mb{A}\iter{\mb{x}}{k+1}-\mb{u}-\iter{\mb{d}}{k}}{2}^2} \\
	\iter{\mb{d}}{k+1}
	=
	\iter{\mb{d}}{k}-\mb{A}\iter{\mb{x}}{k+1}+\iter{\mb{u}}{k+1} \, .
	\end{cases}
\end{equation}
\subsection{Convergence properties with inexact updates} \label{subsec:jour-14-fxr:conv_pro}
The linearized AL method \eqref{eq:jour-14-fxr:general_lal_iterates} is convergent for any fixed AL penalty parameter $\rho>0$ for any $\mb{A}$ \cite{lin:11:lad,wang:12:tla,yang:13:lal,xiao:13:sal}, while the standard AL method is convergent if $\mb{A}$ has full column rank \cite[Theorem 8]{eckstein:92:otd}. Furthermore, even if the AL penalty parameter varies every iteration, \eqref{eq:jour-14-fxr:general_lal_iterates} is convergent when $\rho$ is non-decreasing and bounded above \cite{lin:11:lad}. However, all these convergence analyses assume that all updates are exact. In this paper, we are more interested in the linearized AL method with inexact updates. Specifically, instead of the exact linearized AL method \eqref{eq:jour-14-fxr:general_lal_iterates}, we focus on inexact linearized AL methods:
\begin{equation} \label{eq:jour-14-fxr:inexact_lalm_type_1}
	\begin{cases}
	\Bnorm{
	\iter{\mb{x}}{k+1}
	-
	\nbargmin
	{\mb{x}}
	{\fx{\phi_k}{\mb{x}}}
	}{}
	\leq
	\delta_k \\
	\iter{\mb{u}}{k+1}
	\in
	\argmin
	{\mb{u}}
	{\fx{g}{\mb{u}}+\ts\frac{\rho}{2}\norm{\mb{A}\iter{\mb{x}}{k+1}-\mb{u}-\iter{\mb{d}}{k}}{2}^2} \\
	\iter{\mb{d}}{k+1}
	=
	\iter{\mb{d}}{k}-\mb{A}\iter{\mb{x}}{k+1}+\iter{\mb{u}}{k+1} \, ,
	\end{cases}
\end{equation}
where $\phi_k$ was defined in \eqref{eq:jour-14-fxr:general_lal_iterates}, and
\begin{equation} \label{eq:jour-14-fxr:inexact_lalm_type_2}
	\begin{cases}
	\Big|
	\bfx{\phi_k}{\iter{\mb{x}}{k+1}}
	-
	\nbvarmin
	{\mb{x}}
	{\fx{\phi_k}{\mb{x}}}
	\Big|
	\leq
	\varepsilon_k \\
	\iter{\mb{u}}{k+1}
	\in
	\argmin
	{\mb{u}}
	{\fx{g}{\mb{u}}+\ts\frac{\rho}{2}\norm{\mb{A}\iter{\mb{x}}{k+1}-\mb{u}-\iter{\mb{d}}{k}}{2}^2} \\
	\iter{\mb{d}}{k+1}
	=
	\iter{\mb{d}}{k}-\mb{A}\iter{\mb{x}}{k+1}+\iter{\mb{u}}{k+1} \, .
	\end{cases}
\end{equation}
The $\mb{u}$-update can also be inexact; however, for simplicity, we focus on exact updates of $\mb{u}$. Considering an inexact update of $\mb{u}$ is a trivial extension.

Our convergence analysis of the inexact linearized AL method is twofold. First, we show that the equivalent proximal-point variant of the standard AL iterates \eqref{eq:jour-14-fxr:general_lal_iterates_as_prox_al} can be interpreted as a convergent ADMM that solves another equivalent constrained minimization problem of \eqref{eq:jour-14-fxr:comp_conv_opt_prob} with a redundant split (the proof is in the supplementary material):
\begin{multline} \label{eq:jour-14-fxr:eq_comp_conv_opt_prob_redundant}
	\left(\hat{\mb{x}},\hat{\mb{u}},\hat{\mb{v}}\right)
	\in
	\bargmin{\mb{x},\mb{u},\mb{v}}{\fx{g}{\mb{u}}+\fx{h}{\mb{x}}} \\
	\text{ s.t. }
	\mb{u}=\mb{Ax} \text{ and } \mb{v}=\mb{G}^{1/2}\mb{x}\, .
\end{multline}
Therefore, the linearized AL method is a convergent ADMM, and it has all the nice properties of ADMM, including the tolerance of inexact updates \cite[Theorem 8]{eckstein:92:otd}. More formally, we have the following theorem:
\begin{theorem} \label{thm:jour-14-fxr:inexact_lalm_type_1}
Consider a constrained composite convex optimization problem \eqref{eq:jour-14-fxr:eq_comp_conv_opt_prob} where both $g$ and $h$ are closed and proper convex functions. Let $\rho>0$ and $\left\{\delta_k\right\}_{k=0}^{\infty}$ denote a non-negative sequence such that
\begin{equation} \label{eq:jour-14-fxr:error_bound_type_1}
	\sum_{k=0}^{\infty}\delta_k<\infty \, .
\end{equation}
If \eqref{eq:jour-14-fxr:eq_comp_conv_opt_prob} has a solution $\left(\hat{\mb{x}},\hat{\mb{u}}\right)$, then the sequence of updates $\left\{\left(\iter{\mb{x}}{k},\iter{\mb{u}}{k}\right)\right\}_{k=0}^{\infty}$ generated by the inexact linearized AL method \eqref{eq:jour-14-fxr:inexact_lalm_type_1} converges to $\left(\hat{\mb{x}},\hat{\mb{u}}\right)$; otherwise, at least one of the sequences $\left\{\left(\iter{\mb{x}}{k},\iter{\mb{u}}{k}\right)\right\}_{k=0}^{\infty}$ or $\left\{\iter{\mb{d}}{k}\right\}_{k=0}^{\infty}$ diverges.
\end{theorem}
Theorem \ref{thm:jour-14-fxr:inexact_lalm_type_1} shows that the inexact linearized AL method \eqref{eq:jour-14-fxr:inexact_lalm_type_1} converges if the error $\delta_k$ is absolutely summable. However, it does not describe how fast the algorithm converges and more importantly, how inexact updates affect the convergence rate. This leads to the second part of our convergence analysis.

In this part, we rely on the equivalence between the linearized AL method and the Chambolle-Pock first-order primal-dual algorithm (CPPDA) \cite{chambolle:11:afo}. Consider a minimax problem:
\begin{equation} \label{eq:jour-14-fxr:eq_comp_conv_opt_prob_minimax}
	\left(\hat{\mb{z}},\hat{\mb{x}}\right)
	\in
	\nbargminmax
	{\mb{z}}{\mb{x}}
	{\fx{\Omega}{\mb{z},\mb{x}}} \, ,
\end{equation}
where
\begin{equation} \label{eq:jour-14-fxr:def_minimax_fx}
	\fx{\Omega}{\mb{z},\mb{x}}
	\teq
	\iprod{-\mb{A}'\mb{z}}{\mb{x}}+\fx{g^*}{\mb{z}}-\fx{h}{\mb{x}} \, ,
\end{equation}
and $f^*$ denotes the convex conjugate of a function $f$. Note that $g^{**}=g$ and $h^{**}=h$ since both $g$ and $h$ are closed, proper, and convex. The sequence of updates $\left\{\left(\iter{\mb{z}}{k},\iter{\mb{x}}{k}\right)\right\}_{k=0}^{\infty}$ generated by the CPPDA iterates:
\begin{equation} \label{eq:jour-14-fxr:general_cppd_iterates}
	\begin{cases}
	\iter{\mb{x}}{k+1}
	\in
	\prox{\sigma h}{\iter{\mb{x}}{k}-\sigma\mb{A}'\iter{\bar{\mb{z}}}{k}} \\
	\iter{\mb{z}}{k+1}
	\in
	\prox{\tau g^*}{\iter{\mb{z}}{k}+\tau\mb{A}\iter{\mb{x}}{k+1}} \\
	\iter{\bar{\mb{z}}}{k+1}
	=
	\iter{\mb{z}}{k+1}+\left(\iter{\mb{z}}{k+1}-\iter{\mb{z}}{k}\right)
	\end{cases}
\end{equation}
converges to a saddle-point $\left(\hat{\mb{z}},\hat{\mb{x}}\right)$ of \eqref{eq:jour-14-fxr:eq_comp_conv_opt_prob_minimax}, and the non-negative primal-dual gap $\bfx{\Omega}{\mb{z}_k,\hat{\mb{x}}}-\bfx{\Omega}{\hat{\mb{z}},\mb{x}_k}$ converges to zero with rate $\fx{O}{1/k}$ \cite[Theorem 1]{chambolle:11:afo}, where $\mb{x}_k$ and $\mb{z}_k$ denote the arithemetic mean of all previous $\mb{x}$- and $\mb{z}$-updates up to the $k$th iteration, respectively. Since the CPPDA iterates \eqref{eq:jour-14-fxr:general_cppd_iterates} solve the minimax problem \eqref{eq:jour-14-fxr:eq_comp_conv_opt_prob_minimax}, they also solve the primal problem:
\begin{equation} \label{eq:jour-14-fxr:eq_comp_conv_opt_prob_minimax_primal}
	\hat{\mb{z}}
	\in
	\argmin{\mb{z}}{\fx{h^*}{-\mb{A}'\mb{z}}+\fx{g^*}{\mb{z}}}
\end{equation}
and the dual problem:
\begin{equation} \label{eq:jour-14-fxr:eq_comp_conv_opt_prob_minimax_dual}
	\hat{\mb{x}}
	\in
	\argmax{\mb{x}}{-\fx{g}{\mb{Ax}}-\fx{h}{\mb{x}}}
\end{equation}
of \eqref{eq:jour-14-fxr:eq_comp_conv_opt_prob_minimax}, and the latter happens to be the composite convex optimization problem \eqref{eq:jour-14-fxr:comp_conv_opt_prob}. Therefore, the CPPDA iterates \eqref{eq:jour-14-fxr:general_cppd_iterates} solve \eqref{eq:jour-14-fxr:comp_conv_opt_prob} with rate $\fx{O}{1/k}$ in an ergodic sense. Furthermore, Chambolle \textit{et al.} showed that their proposed primal-dual algorithm is equivalent to a preconditioned ADMM solving \eqref{eq:jour-14-fxr:eq_comp_conv_opt_prob} with a preconditioner $\mb{M}\teq\sigma^{-1}\mb{I}-\tau\mb{A}'\mb{A}$ provided that $0<\sigma\tau<1/\norm{\mb{A}}{2}^2$ \cite[Section 4.3]{chambolle:11:afo}. Letting $\iter{\mb{z}}{k}=-\tau\iter{\mb{d}}{k}$ and choosing $\sigma=\rho^{-1}t$ and $\tau=\rho$, the CPPDA iterates \eqref{eq:jour-14-fxr:general_cppd_iterates} reduce to \eqref{eq:jour-14-fxr:general_lal_iterates_as_prox_al} and hence, the linearized AL method \eqref{eq:jour-14-fxr:general_lal_iterates}. This suggests that we can measure the convergence rate of the linearized AL method using the primal-dual gap that is vanishing ergodically with rate $\fx{O}{1/k}$. Finally, to take inexact updates into account, we apply the error analysis technique developed in \cite{schmidt:11:cro} to the convergence rate analysis of CPPDA, leading to the following theorem (the proof is in the supplementary material):
\begin{theorem} \label{thm:jour-14-fxr:inexact_lalm_type_2}
Consider a minimax problem \eqref{eq:jour-14-fxr:eq_comp_conv_opt_prob_minimax} where both $g$ and $h$ are closed and proper convex functions. Suppose it has a saddle-point $\left(\hat{\mb{z}},\hat{\mb{x}}\right)$, where $\hat{\mb{z}}$ and $\hat{\mb{x}}$ are the solutions of the primal problem \eqref{eq:jour-14-fxr:eq_comp_conv_opt_prob_minimax_primal} and the dual problem \eqref{eq:jour-14-fxr:eq_comp_conv_opt_prob_minimax_dual} of \eqref{eq:jour-14-fxr:eq_comp_conv_opt_prob_minimax}, respectively. Let $\rho>0$ and $\left\{\varepsilon_k\right\}_{k=0}^{\infty}$ denote a non-negative sequence such that
\begin{equation} \label{eq:jour-14-fxr:error_bound_type_2}
	\sum_{k=0}^{\infty}\sqrt{\varepsilon_k}<\infty \, .
\end{equation}
Then, the sequence of updates $\left\{\left(-\rho\iter{\mb{d}}{k},\iter{\mb{x}}{k}\right)\right\}_{k=0}^{\infty}$ generated by the inexact linearized AL method \eqref{eq:jour-14-fxr:inexact_lalm_type_2} is a bounded sequence that converges to $\left(\hat{\mb{z}},\hat{\mb{x}}\right)$, and the primal-dual gap of $\left(\mb{z}_k,\mb{x}_k\right)$ has the following bound:
\begin{equation} \label{eq:jour-14-fxr:primal_dual_gap_bound_with_error}
	\bfx{\Omega}{\mb{z}_k,\hat{\mb{x}}}-\bfx{\Omega}{\hat{\mb{z}},\mb{x}_k}
	\leq
	\frac{\left(C+2A_k+\sqrt{B_k}\right)^2}{k} \, ,
\end{equation}
where $\mb{z}_k\teq\frac{1}{k}\sum_{j=1}^k\big({-\rho}\iter{\mb{d}}{j}\big)$, $\mb{x}_k\teq\frac{1}{k}\sum_{j=1}^k\iter{\mb{x}}{j}$,
\begin{equation} \label{eq:jour-14-fxr:def_C}
	C\teq\frac{\norm{\iter{\mb{x}}{0}-\hat{\mb{x}}}{2}}{\sqrt{2\rho^{-1}t}}+\frac{\norm{\big({-\rho}\iter{\mb{d}}{0}\big)-\hat{\mb{z}}}{2}}{\sqrt{2\rho}} \, ,
\end{equation}
\begin{equation} \label{eq:jour-14-fxr:def_Ak}
	A_k\teq\sum_{j=1}^k\sqrt{\frac{\varepsilon_{j-1}}{\big(1-t\norm{\mb{A}}{2}^2\big)\rho^{-1}t}} \, ,
\end{equation}
and
\begin{equation} \label{eq:jour-14-fxr:def_Bk}
	B_k\teq\sum_{j=1}^k\varepsilon_{j-1} \, .
\end{equation}
\end{theorem}
Theorem \ref{thm:jour-14-fxr:inexact_lalm_type_2} shows that the inexact linearized AL method \eqref{eq:jour-14-fxr:inexact_lalm_type_2} converges with rate $\fx{O}{1/k}$ if the square root of the error $\varepsilon_k$ is absolutely summable. In fact, even if $\left\{\sqrt{\varepsilon_k}\right\}_{k=0}^{\infty}$ is not absolutely summable, say, $\sqrt{\varepsilon_k}$ decreases as $\fx{O}{1/k}$, $A_k$ grows as $\fx{O}{\log k}$ (note that $B_k$ always grows slower than $A_k$), and the primal-dual gap converges to zero in $\fx{O}{\log^2 k/k}$. To obtain convergence of the primal-dual gap, a necessary condition is that the partial sum of $\left\{\sqrt{\varepsilon_k}\right\}_{k=0}^{\infty}$ grows no faster than $\bfx{o}{\sqrt{k}}$.

The primal-dual gap convergence bound above is measured at the average point $\left({-\rho}\mb{d}_k,\mb{x}_k\right)$ of the update trajectory. In practice, the primal-dual gap of $\left({-\rho}\iter{\mb{d}}{k},\iter{\mb{x}}{k}\right)$ converges much faster than that. Minimizing the constant in \eqref{eq:jour-14-fxr:primal_dual_gap_bound_with_error} need not provide the fastest convergence rate of the linearized AL method. However, the $\rho$-, $t$-, and $\varepsilon_k$-dependence in \eqref{eq:jour-14-fxr:primal_dual_gap_bound_with_error} suggests how these factors affect the convergence rate of the linearized AL method. Finally, although we consider only one variable split in our derivation, it is easy to extend our proofs to support multiple variable splits by using the variable splitting scheme in \cite{afonso:11:aal}. Hence, when $M=1$ and $g$ has a simple proximal mapping, Theorem \ref{thm:jour-14-fxr:inexact_lalm_type_2} suggests that the linearized AL method might be more efficient than stochastic ADMM \cite{ouyang:13:sad} because no growing diagonal majorizer is required in the linearized AL method. However, unlike stochastic ADMM, the linearized AL method is not OS-accelerable in general because the $\mb{d}$-update takes a full forward projection. We use the linearized AL method for analysis and to motivate the proposed algorithm in Section~\ref{sec:jour-14-fxr:proposed_alg}, but it is not recommended for practical implementation in CT reconstruction. By restricting $g$ to be a quadratic loss function, we show that the linearized AL method becomes OS-accelerable.

\section{Proposed algorithm} \label{sec:jour-14-fxr:proposed_alg}
\subsection{\mbox{OS-LALM}: an OS-accelerable splitting-based algorithm} \label{subsec:jour-14-fxr:os_lalm}
In this section, we restrict $g$ to be a quadratic loss function, i.e., $\fx{g}{\mb{u}}\teq\frac{1}{2}\norm{\mb{y}-\mb{u}}{2}^2$, and then the minimization problem \eqref{eq:jour-14-fxr:comp_conv_opt_prob} becomes a regularized least-squares problem:
\begin{equation} \label{eq:jour-14-fxr:penalized_ls_prob}
	\hat{\mb{x}}
	\in
	\argmin{\mb{x}}{\fx{\Psi}{\mb{x}}\teq\ts\frac{1}{2}\norm{\mb{y}-\mb{Ax}}{2}^2+\fx{h}{\mb{x}}} \, .
\end{equation}
Let $\fx{\ell}{\mb{x}}\teq\fx{g}{\mb{Ax}}$ denote the quadratic data-fitting term in \eqref{eq:jour-14-fxr:penalized_ls_prob}. We assume that $\ell$ is suitable for OS acceleration; i.e., $\ell$ can be decomposed into $M$ smaller quadratic functions $\ell_1,\ldots,\ell_M$ satisfying the ``subset balance condition'' \cite{erdogan:99:osa}:
\begin{equation} \label{eq:jour-14-fxr:subset_balance_cond}
	\fx{\nabla \ell}{\mb{x}}\approx M\fx{\nabla \ell_1}{\mb{x}}\approx\cdots\approx M\fx{\nabla \ell_M}{\mb{x}} \, ,
\end{equation}
so that the subset gradients approximate the full gradient of $\ell$.

Since $g$ is quadratic, its proximal mapping is linear. The $\mb{u}$-update in the linearized AL method \eqref{eq:jour-14-fxr:general_lal_iterates} has the following simple closed-form solution:
\begin{equation} \label{eq:jour-14-fxr:lalm_pls_u}
	\iter{\mb{u}}{k+1}
	=
	\ts\frac{\rho}{\rho+1}\big(\mb{A}\iter{\mb{x}}{k+1}-\iter{\mb{d}}{k}\big)
	+
	\ts\frac{1}{\rho+1}\mb{y} \, .
\end{equation}
Combining \eqref{eq:jour-14-fxr:lalm_pls_u} with the $\mb{d}$-update of \eqref{eq:jour-14-fxr:general_lal_iterates} yields the identity
\begin{equation} \label{eq:jour-14-fxr:u_d_identity}
	\iter{\mb{u}}{k+1}+\rho\iter{\mb{d}}{k+1}=\mb{y}
\end{equation}
if we initialize $\mb{d}$ as $\iter{\mb{d}}{0}=\rho^{-1}\left(\mb{y}-\iter{\mb{u}}{0}\right)$. Letting $\tilde{\mb{u}}\teq\mb{u}-\mb{y}$ denote the split residual and substituting \eqref{eq:jour-14-fxr:u_d_identity} into \eqref{eq:jour-14-fxr:general_lal_iterates} leads to the following simplified linearized AL iterates:
\begin{equation} \label{eq:jour-14-fxr:simplified_lal_iterates}
	\begin{cases}
	\iter{\mb{s}}{k+1}
	=
	\mb{A}'\big(\rho\big(\mb{A}\iter{\mb{x}}{k}-\mb{y}\big)+\left(1-\rho\right)\iter{\tilde{\mb{u}}}{k}\big) \\
	\iter{\mb{x}}{k+1}
	\in
	\prox{\left(\rho^{-1}t\right)h}{\iter{\mb{x}}{k}-(\rho^{-1}t)\,\iter{\mb{s}}{k+1}} \\
	\iter{\tilde{\mb{u}}}{k+1}
	=
	\ts\frac{\rho}{\rho+1}\big(\mb{A}\iter{\mb{x}}{k+1}-\mb{y}\big)
	+
	\ts\frac{1}{\rho+1}\iter{\tilde{\mb{u}}}{k} \, .
	\end{cases}
\end{equation}
The net computational complexity of \eqref{eq:jour-14-fxr:simplified_lal_iterates} per iteration reduces to one multiplication by $\mb{A}$, one multiplication by $\mb{A}'$, and one proximal mapping of $h$ that often can be solved non-iteratively or solved iteratively without using $\mb{A}$ or $\mb{A}'$. Since the gradient of $\ell$ is $\mb{A}'\left(\mb{Ax}-\mb{y}\right)$, letting $\mb{g}\teq\mb{A}'\tilde{\mb{u}}$ (a back-projection of the split residual) denote the split gradient, we can rewrite \eqref{eq:jour-14-fxr:simplified_lal_iterates} as:
\begin{equation} \label{eq:jour-14-fxr:lalm_iterates}
	\begin{cases}
	\iter{\mb{s}}{k+1}
	=
	\rho\bfx{\nabla\ell}{\iter{\mb{x}}{k}}
	+
	\left(1-\rho\right)\iter{\mb{g}}{k}\\
	\iter{\mb{x}}{k+1}
	\in
	\prox{\left(\rho^{-1}t\right)h}{\iter{\mb{x}}{k}-(\rho^{-1}t)\,\iter{\mb{s}}{k+1}} \\
	\iter{\mb{g}}{k+1}
	=
	\ts\frac{\rho}{\rho+1}\bfx{\nabla\ell}{\iter{\mb{x}}{k+1}}
	+
	\ts\frac{1}{\rho+1}\iter{\mb{g}}{k} \, .
	\end{cases}
\end{equation}
We call \eqref{eq:jour-14-fxr:lalm_iterates} the gradient-based linearized AL method because only the gradients of $\ell$ are used to perform the updates, and the net computational complexity of \eqref{eq:jour-14-fxr:lalm_iterates} per iteration becomes one gradient evaluation of $\ell$ and one proximal mapping of $h$.

We interpret the gradient-based linearized AL method \eqref{eq:jour-14-fxr:lalm_iterates} as a generalized proximal gradient descent of a regularized least-squares cost function $\Psi$ with step size $\rho^{-1}t$ and search direction $\iter{\mb{s}}{k+1}$ that is a linear average of the gradient and split gradient of $\ell$. A smaller $\rho$ can lead to a larger step size. When $\rho=1$, \eqref{eq:jour-14-fxr:lalm_iterates} happens to be the proximal gradient method or the iterative shrinkage/thresholding algorithm (ISTA) \cite{daubechies:04:ait}. In other words, by using the linearized AL method, we can arbitrarily increase the step size of the proximal gradient method by decreasing $\rho$, thanks to the simple $\rho$-dependent correction of the search direction in \eqref{eq:jour-14-fxr:lalm_iterates}. To have a concrete example, suppose all updates are exact, i.e., $\varepsilon_k=0$ for all $k$. From \eqref{eq:jour-14-fxr:u_d_identity} and Theorem \ref{thm:jour-14-fxr:inexact_lalm_type_2}, we have $-\rho\iter{\mb{d}}{k}=\iter{\mb{u}}{k}-\mb{y}\to\mb{A}\hat{\mb{x}}-\mb{y}=\hat{\mb{z}}$ as $k\to\infty$. Furthermore, $\left(-\rho\iter{\mb{d}}{0}\right)-\hat{\mb{z}}=\iter{\mb{u}}{0}-\mb{A}\hat{\mb{x}}$. Therefore, with a reasonable initialization, e.g., $\iter{\mb{u}}{0}=\mb{A}\iter{\mb{x}}{0}$ and consequently, $\iter{\mb{g}}{0}=\bfx{\nabla\ell}{\iter{\mb{x}}{0}}$, the constant $C$ in \eqref{eq:jour-14-fxr:def_C} can be rewritten as a function of $\rho$:
\begin{equation} \label{eq:jour-14-fxr:def_C_new}
	\fx{C}{\rho}
	=
	\frac{\norm{\iter{\mb{x}}{0}-\hat{\mb{x}}}{2}}{\sqrt{2\rho^{-1}t}}
	+
	\frac{\norm{\mb{A}\big(\iter{\mb{x}}{0}-\hat{\mb{x}}\big)}{2}}{\sqrt{2\rho}} \, .
\end{equation}
This constant achieves its minimum at
\begin{equation} \label{eq:jour-14-fxr:opt_rho_minimize_C}
	\rho_{\text{opt}}
	=
	\frac{\norm{\mb{A}\big(\iter{\mb{x}}{0}-\hat{\mb{x}}\big)}{2}}{\sqrt{L}\norm{\iter{\mb{x}}{0}-\hat{\mb{x}}}{2}}
	\leq
	1 \, ,
\end{equation}
and it suggests that unity might be a reasonable upper bound on $\rho$ for fast convergence. When the majorization is loose, i.e., $L\gg\norm{\mb{A}}{2}^2$, then $\rho_{\text{opt}}\ll 1$. In this case, the first term in \eqref{eq:jour-14-fxr:def_C_new} dominates $C$ for $\rho_{\text{opt}}<\rho\leq 1$, and the upper bound of the primal-dual gap becomes
\begin{equation} \label{eq:jour-14-fxr:primal_dual_gap_bound_with_error_new}
	\bfx{\Omega}{\mb{z}_k,\hat{\mb{x}}}-\bfx{\Omega}{\hat{\mb{z}},\mb{x}_k}
	\leq
	\frac{C^2}{k}
	\approx
	\fx{O}{\frac{1}{\rho^{-1}k}} \, .
\end{equation}
That is, comparing to the proximal gradient method ($\rho=1$), the convergence rate (bound) of our proposed algorithm is accelerated by a factor of $\rho^{-1}$ for $\rho_{\text{opt}}<\rho\leq 1$!

Finally, since the proposed gradient-based linearized AL method \eqref{eq:jour-14-fxr:lalm_iterates} requires only the gradients of $\ell$ to perform the updates, it is OS-accelerable! For OS acceleration, we simply replace $\nabla\ell$ in \eqref{eq:jour-14-fxr:lalm_iterates} with $M\nabla\ell_m$ using the approximation \eqref{eq:jour-14-fxr:subset_balance_cond} and incrementally perform \eqref{eq:jour-14-fxr:lalm_iterates} for $M$ times as a complete iteration, thus leading to the final proposed OS-accelerable linearized AL method (\mbox{OS-LALM}):
\begin{equation} \label{eq:jour-14-fxr:os_lalm_iterates}
	\begin{cases}
	\iter{\mb{s}}{k,m+1}
	=
	\rho M\bfx{\nabla\ell_m}{\iter{\mb{x}}{k,m}}
	+
	\left(1-\rho\right)\iter{\mb{g}}{k,m}\\
	\iter{\mb{x}}{k,m+1}
	\in
	\prox{\left(\rho^{-1}t\right)h}{\iter{\mb{x}}{k,m}-(\rho^{-1}t)\,\iter{\mb{s}}{k,m+1}} \\
	\iter{\mb{g}}{k,m+1}
	=
	\ts\frac{\rho}{\rho+1}M\bfx{\nabla\ell_{m+1}}{\iter{\mb{x}}{k,m+1}}
	+
	\ts\frac{1}{\rho+1}\iter{\mb{g}}{k,m}
	\end{cases}
\end{equation}
with $\iter{\mb{c}}{k,M+1}=\iter{\mb{c}}{k+1}=\iter{\mb{c}}{k+1,1}$ for $\mb{c}\in\left\{\mb{s},\mb{x},\mb{g}\right\}$ and $\ell_{M+1}=\ell_1$. Like typical OS-based algorithms, this algorithm is convergent when $M=1$, i.e., \eqref{eq:jour-14-fxr:lalm_iterates}, but is not guaranteed to converge for $M>1$. When $M>1$, updates generated by OS-based algorithms enter a ``limit cycle'' in which updates stop approaching the optimum, and visible OS artifacts might be observed in the reconstructed image, depending on $M$.
\subsection{Deterministic downward continuation} \label{subsec:jour-14-fxr:down_cont}
One drawback of the AL method with a fixed AL penalty parameter $\rho$ is the difficulty of finding the value that provides the fastest convergence. For example, although the optimal AL penalty parameter $\rho_{\text{opt}}$ in \eqref{eq:jour-14-fxr:opt_rho_minimize_C} minimizes the constant $C$ in \eqref{eq:jour-14-fxr:def_C_new} that governs the convergence rate of the primal-dual gap, one cannot know its value beforehand because it depends on the solution $\hat{\mb{x}}$ of the problem. Intuitively, a smaller $\rho$ is better because it leads to a larger step size. However, when the step size is too large, one can encounter overshoots and oscillations that slow down the convergence rate at first and when nearing the optimum. In fact, $\rho_\text{opt}$ in \eqref{eq:jour-14-fxr:opt_rho_minimize_C} also suggests that $\rho$ should not be arbitrarily small. Rather than estimating $\rho_{\text{opt}}$ heuristically, we focus on using an iteration-dependent $\rho$, i.e., a continuation approach, for acceleration.

The classic continuation approach increases $\rho$ as the algorithm proceeds so that the previous iterate can serve as a warm start for the subsequent worse-conditioned but more penalized inner minimization problem \cite[Proposition 4.2.1]{bertsekas:99:np}. However, in classic continuation approaches such as \cite{lin:11:lad}, one must specify both the initial value and the update rules of $\rho$. This introduces even more parameters to be tuned. In this paper, unlike classic continuation approaches, we consider a downward continuation approach. The intuition is that, for a fixed $\rho$, the step length $\big\|\iter{\mb{x}}{k+1}-\iter{\mb{x}}{k}\big\|$ is typically a decreasing sequence because the gradient norm vanishes as we approach the optimum, and an increasing sequence $\rho_k$ (i.e., a diminishing step size) would aggravate the shrinkage of step length and slow down the convergence rate. In contrast, a decreasing sequence $\rho_k$ can compensate for step length shrinkage and accelerate convergence. Of course, $\rho_k$ cannot decrease too fast; otherwise, the soaring step size might make the algorithm unstable or even divergent. To design a ``good'' decreasing sequence $\rho_k$ for ``effective'' acceleration, we first analyze how our proposed algorithm (the one-subset version \eqref{eq:jour-14-fxr:lalm_iterates} for simplicity) behaves for different values of $\rho$.

Consider a very simple quadratic problem:
\begin{equation} \label{eq:jour-14-fxr:quad_problem}
	\hat{\mb{x}}
	\in
	\nbargmin{\mb{x}}{\ts\frac{1}{2}\norm{\mb{Ax}}{2}^2} \, .
\end{equation}
It is just an instance of \eqref{eq:jour-14-fxr:penalized_ls_prob} with $h=0$ and $\mb{y}=\mb{0}$. One trivial solution of \eqref{eq:jour-14-fxr:quad_problem} is $\hat{\mb{x}}=\mb{0}$. To ensure a unique solution, we assume that $\mb{A}'\mb{A}$ is positive definite (for this analysis only). Let $\mb{A}'\mb{A}$ have eigenvalue decomposition $\mb{V}\msb{\Lambda}\mb{V}'$, where \mbox{$\msb{\Lambda}\teq\diag{\lambda_i}$} and $\mu=\lambda_1\leq\dots\leq\lambda_n=L$. The updates generated by \eqref{eq:jour-14-fxr:lalm_iterates} that solve \eqref{eq:jour-14-fxr:quad_problem} can be written as
\begin{equation} \label{eq:jour-14-fxr:lalm_iterates_qp}
	\begin{cases}
	\iter{\mb{x}}{k+1}
	=
	\iter{\mb{x}}{k}-(1/L)\big(\mb{V}\msb{\Lambda}\mb{V}'\iter{\mb{x}}{k}+(\rho^{-1}-1)\,\iter{\mb{g}}{k}\big) \\
	\iter{\mb{g}}{k+1}
	=
	\frac{\rho}{\rho+1}\mb{V}\msb{\Lambda}\mb{V}'\iter{\mb{x}}{k+1}
	+
	\frac{1}{\rho+1}\iter{\mb{g}}{k} \, .
	\end{cases}
\end{equation}
Furthermore, letting $\bar{\mb{x}}=\mb{V}'\mb{x}$ and $\bar{\mb{g}}=\mb{V}'\mb{g}$, the linear system can be further diagonalized, and we can represent the $i$th components of $\bar{\mb{x}}$ and $\bar{\mb{g}}$ as
\begin{equation} \label{eq:jour-14-fxr:ith_lalm_iterates_qp}
	\begin{cases}
	\iter{\bar{x}_i}{k+1}
	=
	\iter{\bar{x}_i}{k}-(1/L)\big(\lambda_i\iter{\bar{x}_i}{k}+(\rho^{-1}-1)\,\iter{\bar{g}_i}{k}\big) \\
	\iter{\bar{g}_i}{k+1}
	=
	\frac{\rho}{\rho+1}\lambda_i\iter{\bar{x}_i}{k+1}
	+
	\frac{1}{\rho+1}\iter{\bar{g}_i}{k} \, .
	\end{cases}
\end{equation}
Solving this system of recurrence relations of $\bar{x}_i$ and $\bar{g}_i$, one can show that both $\bar{x}_i$ and $\bar{g}_i$ satisfy a second-order recursive system determined by the characteristic polynomial:
\begin{equation} \label{eq:jour-14-fxr:char_poly_qp}
	\left(1+\rho\right)r^2-2\left(1-\lambda_i/L+\rho/2\right)r+\left(1-\lambda_i/L\right) \, .
\end{equation}
We analyze the roots of this polynomial to examine the convergence rate.

When $\rho=\rho_i^{\text{c}}$, where
\begin{equation} \label{eq:jour-14-fxr:critical_value_qp}
	\rho_i^{\text{c}}
	\teq
	2\sqrt{\frac{\lambda_i}{L}\left(1-\frac{\lambda_i}{L}\right)}
	\in
	(0,1] \, ,
\end{equation}
the characteristic equation has repeated roots. Hence, the system is critically damped, and $\bar{x}_i$ and $\bar{g}_i$ converge linearly to zero with convergence rate
\begin{equation} \label{eq:jour-14-fxr:critical_damp_rate_qp}
	r_i^{\text{c}}
	=
	\frac{1-\lambda_i/L+\rho_i^{\text{c}}/2}{1+\rho_i^{\text{c}}}
	=
	\sqrt{\frac{1-\lambda_i/L}{1+\rho_i^{\text{c}}}} \, .
\end{equation}
When $\rho>\rho_i^{\text{c}}$, the characteristic equation has distinct real roots. Hence, the system is over-damped, and $\bar{x}_i$ and $\bar{g}_i$ converge linearly to zero with convergence rate that is governed by the dominant root
\begin{equation} \label{eq:jour-14-fxr:over_damp_rate_qp}
	\fx{r_i^{\text{o}}}{\rho}
	=
	\frac{1-\lambda_i/L+\rho/2+\sqrt{\rho^2/4-\lambda_i/L\left(1-\lambda_i/L\right)}}{1+\rho} \, .
\end{equation}
It is easy to check that $\fx{r_i^{\text{o}}}{\rho_i^{\text{c}}}=r_i^{\text{c}}$, and $r_i^{\text{o}}$ is non-decreasing. This suggests that the critically damped system always converges faster than the over-damped system. Finally, when $\rho<\rho_i^{\text{c}}$, the characteristic equation has complex roots. In this case, the system is under-damped, and $\bar{x}_i$ and $\bar{g}_i$ converge linearly to zero with convergence rate
\begin{equation} \label{eq:jour-14-fxr:under_damp_rate_qp}
	\fx{r_i^{\text{u}}}{\rho}
	=
	\frac{1-\lambda_i/L+\rho/2}{1+\rho} \, ,
\end{equation}
and oscillate at the damped frequency $\psi_i/(2\pi)$, where
\begin{equation} \label{eq:jour-14-fxr:under_damp_cosine_qp}
	\text{cos}\psi_i
	=
	\frac{1-\lambda_i/L+\rho/2}{\sqrt{(1+\rho)(1-\lambda_i/L)}}
	\approx
	\sqrt{1-\lambda_i/L}
\end{equation}
when $\rho\approx0$. Furthermore, by the small angle approximation: $\text{cos}\sqrt{\theta}\approx1-\theta/2\approx\sqrt{1-\theta}$, if $\lambda_i\gg L$, $\psi_i\approx\sqrt{\lambda_i/L}$. Again, $\fx{r_i^{\text{u}}}{\rho_i^{\text{c}}}=r_i^{\text{c}}$, but $r_i^{\text{u}}$ behaves differently from $r_i^{\text{o}}$. Specifically, $r_i^{\text{u}}$ is non-increasing if $\lambda_i/L<1/2$, and it is non-decreasing otherwise. This suggests that the critically damped system converges faster than the under-damped system if $\lambda_i/L<1/2$, but it can be slower otherwise. In sum, the critically damped system is optimal for those eigencomponents with smaller eigenvalues (i.e., $\lambda_i<L/2$), while for eigencomponents with larger eigenvalues (i.e., $\lambda_i>L/2$), the under-damped system is optimal.

In practice, the asymptotic convergence rate of the system is dominated by the smallest eigenvalue $\lambda_1=\mu$. As the algorithm proceeds, only the component oscillating at the frequency $\psi_1/(2\pi)$ persists. Therefore, to achieve the fastest asymptotic convergence rate, we would like to choose
\begin{equation} \label{eq:jour-14-fxr:opt_rho_qp}
	\rho^{\star}
	=
	\rho_1^{\text{c}}
	=
	2\sqrt{\frac{\mu}{L}\left(1-\frac{\mu}{L}\right)}
	\in
	(0,1] \, .
\end{equation}
Unlike $\rho_{\text{opt}}$ in \eqref{eq:jour-14-fxr:opt_rho_minimize_C}, this choice of $\rho$ does not depend on the initialization. It depends only on the geometry of the Hessian $\mb{A}'\mb{A}$. Furthermore, notice that both $\rho_{\text{opt}}$ and $\rho^{\star}$ fall in the interval $(0,1]$. Hence, although the linearized AL method converges for any $\rho>0$, we consider only $\rho\leq 1$ in our downward continuation approach.

We can now interpret the classic (upward) continuation approach based on the second-order recursive system analysis. The classic continuation approach usually starts from a small $\rho$ for better-conditioned inner minimization problem. Therefore, initially, the system is under-damped. Although the under-damped system has a slower asymptotic convergence rate, the oscillation can provide dramatic acceleration before the first zero-crossing of the oscillating components. We can think of the classic continuation approach as a greedy strategy that exploits the initial fast convergence rate of the under-damped system and carefully increases $\rho$ to avoid oscillation and move toward the critical damping regime. However, this greedy strategy requires a ``clever'' update rule for increasing $\rho$. If $\rho$ increases too fast, the acceleration ends prematurally; if $\rho$ increases too slow, the system starts oscillating.

In contrast, we consider a more conservative strategy that starts from the over-damped regime, say, $\rho=1$ as suggested in \eqref{eq:jour-14-fxr:opt_rho_qp}, and gradually reduces $\rho$ to the optimal AL penalty parameter $\rho^{\star}$. It sounds impractical at first because we do not know $\mu$ beforehand. To solve this problem, we adopt the adaptive restart proposed in \cite{odonoghue:13:arf} and generate a decreasing sequence $\rho_k$ that starts from $\rho=1$ and reaches $\rho^{\star}$ every time the algorithm restarts! As mentioned before, the system oscillates at frequency $\psi_1/(2\pi)$ when it is under-damped. This oscillating behavior can also be observed from the trajectory of updates. For example,
\begin{equation} \label{eq:jour-14-fxr:period_indicator_qp}
	\fx{\xi}{k}
	\teq
	\big(\iter{\mb{g}}{k}-\nabla\ell\big(\iter{\mb{x}}{k+1}\big)\big)'
	\big(\nabla\ell\big(\iter{\mb{x}}{k+1}\big)-\nabla\ell\big(\iter{\mb{x}}{k}\big)\big)
\end{equation}
oscillates at the frequency $\psi_1/\pi$ \cite{odonoghue:13:arf}. Hence, if we restart every time $\fx{\xi}{k}>0$, we restart the decreasing sequence about every $\left(\pi/2\right)\sqrt{L/\mu}$ iterations. Suppose we restart at the $r$th iteration, we have the approximation $\sqrt{\mu/L}\approx\pi/\left(2r\right)$, and the ideal AL penalty parameter at the $r$th iteration should be
\begin{equation} \label{eq:jour-14-fxr:opt_rho_r_qp}
	2\sqrt{\big(\ts\frac{\pi}{2r}\big)^{\sscs2}\big(1-\big(\ts\frac{\pi}{2r}\big)^{\sscs2}\big)}
	=
	\ts\frac{\pi}{r}\sqrt{1-\big(\ts\frac{\pi}{2r}\big)^{\sscs2}} \, .
\end{equation}
Finally, the proposed downward continuation approach has the form \eqref{eq:jour-14-fxr:lalm_iterates}, while we replace every $\rho$ in \eqref{eq:jour-14-fxr:lalm_iterates} with
\begin{equation} \label{eq:jour-14-fxr:rho_l_formula_qp}
	\rho_l
	=
	\begin{cases}
	1 &\!\!\!\text{, if $l=0$} \\
	\text{max}\!\left\{\frac{\pi}{l+1}\sqrt{1-\big(\frac{\pi}{2l+2}\big)^{\sscs2}},\rho_{\text{min}}\right\} &\!\!\!\text{, otherwise} \, ,
	\end{cases}
\end{equation}
where $l$ is a counter that starts from zero, increases by one, and is reset to zero whenever $\fx{\xi}{k}>0$. For the $M$-subset version \eqref{eq:jour-14-fxr:os_lalm_iterates}, we simply replace the gradients with the gradient approximations in \eqref{eq:jour-14-fxr:period_indicator_qp} and check the restart condition every inner iteration. The lower bound $\rho_{\text{min}}$ is a small positive number for guaranteeing convergence. Note that ADMM is convergent if $\rho$ is non-increasing and bounded below away from zero \cite[Corollary 4.2]{kontogiorgis:98:avp}. As shown in Section~\ref{subsec:jour-14-fxr:conv_pro}, the linearized AL method is in fact a convergent ADMM. Therefore, we can ensure convergence (of the one-subset version) of the proposed downward continuation approach if we set a non-zero lower bound for $\rho_l$, e.g., $\rho_{\text{min}}=10^{-3}$ in our experiments.

Note that $\rho_l$ in \eqref{eq:jour-14-fxr:rho_l_formula_qp} is the same for any $\mb{A}$. The adaptive restart condition takes care of the dependence on $\mb{A}$. That is why we call this approach the deterministic downward continuation approach. When $h$ is non-zero and/or $\mb{A}'\mb{A}$ is not positive definite, our analysis above does not hold. However, the deterministic downward continuation approach works well in practice for CT. One possible explanation is that the cost function can usually be well approximated by a quadratic near the optimum when the minimization problem is well-posed and $h$ is locally quadratic.

\section{Implementation details} \label{sec:jour-14-fxr:impl_detail}
In this section, we consider solving the X-ray CT image reconstruction problem:
\begin{equation} \label{eq:jour-14-fxr:ct_recon}
	\hat{\mb{x}}
	\in
	\argmin{\mb{x}\in\Omega}{\ts\frac{1}{2}\norm{\mb{y}-\mb{Ax}}{\mb{W}}^2+\fx{\R}{\mb{x}}}
\end{equation}
using the proposed algorithm, where $\mb{A}$ is the system matrix of a CT scan, $\mb{y}$ is the noisy sinogram, $\mb{W}$ is the statistical weighting matrix, $\R$ is an edge-preserving regularizer, and $\Omega$ denotes the convex set for a box constraint (usually the non-negativity constraint) on $\mb{x}$.
\subsection{\mbox{OS-LALM} for X-ray CT image reconstruction} \label{subsec:jour-14-fxr:os_lalm_ct}
The X-ray CT image reconstruction problem \eqref{eq:jour-14-fxr:ct_recon} is a constrained regularized weighted least-squares problem. To solve it using the proposed algorithm \eqref{eq:jour-14-fxr:lalm_iterates} and its OS variant \eqref{eq:jour-14-fxr:os_lalm_iterates}, we use the following substitution:
\begin{equation} \label{eq:jour-14-fxr:ct_substitution}
	\begin{cases}
	\mb{A}\leftarrow\mb{W}^{1/2}\mb{A} \\
	\mb{y}\leftarrow\mb{W}^{1/2}\mb{y} \\
	h\leftarrow\R+\iota_{\Omega} \, ,
	\end{cases}
\end{equation}
where $\iota_{\mathcal{C}}$ denotes the characteristic function of a convex set $\mathcal{C}$. Thus, the inner minimization problem in \eqref{eq:jour-14-fxr:lalm_iterates} and its OS variant \eqref{eq:jour-14-fxr:os_lalm_iterates} becomes a constrained denoising problem. In our implementation, we solve this inner constrained denoising problem using $n$ iterations of the fast iterative shrinkage/thresholding algorithm (FISTA) \cite{beck:09:afi} starting from the previous update as a warm start. As discussed in Section~\ref{subsec:jour-14-fxr:conv_pro}, inexact updates can slow down the convergence rate of the proposed algorithm. In general, the more FISTA iterations, the faster convergence rate of the proposed algorithm. However, the overhead of iterative inner updates is non-negligible for large $n$, especially when the number of subsets is large. Fortunately, in typical X-ray CT image reconstruction problems, the majorization is usually very loose (probably due to the huge dynamic range of the statistical weighting $\mb{W}$). Therefore, $t\ll1$ in most cases, greatly diminishing the regularization force in the constrained denoising problem. In practice, the constrained denoising problem can be solved up to some acceptable tolerance within just one or two iterations! Finally, for a fair comparison with other OS-based algorithms, in our experiments, we majorize the quadratic penalty in the scaled augmented Lagrangian using the SQS function with Hessian $\mb{L}_{\text{diag}}\teq\diag{\mb{A}'\mb{WA1}}$ \cite{erdogan:99:osa} and incrementally update the image using the subset gradients with the bit-reversal order \cite{herman:93:art} that heuristically minimizes the subset gradient variance as in other OS-based algorithms \cite{erdogan:99:osa,kim:13:aos,kim:13:axr,kim:13:osa}.

By the way, as mentioned before, the SQS function with Hessian $\mb{L}_{\text{diag}}$ is a very loose majorizer. To achieve the fastest convergence, one might want to use the tightest majorizer with Hessian $\mb{A}'\mb{WA}$. However, this would revert to the standard AL method \eqref{eq:jour-14-fxr:general_al_iterates} with expensive $\mb{x}$-updates. An alternative is the Barzilai-Borwein (spectral) method \cite{barzilai:88:tps} that mimics the Hessian $\mb{A}'\mb{WA}$ by $\mb{H}_k\teq\alpha_k\mb{L}_{\text{diag}}$, where the scaling factor $\alpha_k$ is solved by fitting the secant equation in the (weighted) least-squares sense. Detailed derivation and additional experimental results can be found in the supplementary material.
\subsection{Number of subsets} \label{subsec:jour-14-fxr:num_subsets}
As mentioned in Section~\ref{subsec:jour-14-fxr:os_lalm}, the number of subsets $M$ can affect the stability of OS-based algorithms. When $M$ is too large, OS algorithms typically become unstable, and one can observe artifacts in the reconstructed image. Therefore, finding an appropriate number of subsets is very important. Since errors of OS-based algorithms come from the gradient approximation using subset gradients, artifacts might be supressed using a better gradient approximation. Intuitively, to have an acceptable gradient approximation, each voxel in a subset should be sampled by a minimum number of views $s$. For simplicity, we consider the central voxel in the transaxial plane. In axial CT, the views are uniformly distributed in each subset, so we want
\begin{equation} \label{eq:jour-14-fxr:M_constraint_axial}
	\ts\frac{1}{M_\text{axial}}\cdot(\text{number of views})\geq s_{\text{axial}} \, .
\end{equation}
This leads to our maximum number of subsets for axial CT:
\begin{equation} \label{eq:jour-14-fxr:max_M_axial}
	M_{\text{axial}}\leq(\text{number of views})\cdot\ts\frac{1}{s_{\text{axial}}} \, .
\end{equation}
In helical CT, the situation is more complicated. Since the X-ray source moves in the z direction, a central voxel is only covered by $d_{\text{so}}/\left(p\cdot d_{\text{sd}}\right)$ turns, where $p$ is the pitch, $d_{\text{so}}$ denotes the distance from the X-ray source to the isocenter, and $d_{\text{sd}}$ denotes the distance from the X-ray source to the detector. Therefore, we want
\begin{equation} \label{eq:jour-14-fxr:M_constraint_helical}
	\ts\frac{1}{M_\text{helical}} \cdot (\text{number of views per turn}) \cdot \ts\frac{d_{\text{so}}}{p\cdot d_{\text{sd}}}
	\geq
	s_{\text{helical}} \, .
\end{equation}
This leads to our maximum number of subsets for helical CT:
\begin{equation} \label{eq:jour-14-fxr:max_M_helical}
	M_{\text{helical}}\leq(\text{number of views per turn})\cdot\ts\frac{d_{\text{so}}}{p\cdot s_{\text{helical}}\cdot d_{\text{sd}}} \, .
\end{equation}
Note that the maximum number of subsets for helical CT $M_{\text{helical}}$ is inversely proportional to the pitch $p$. That is, the maximum number of subsets for helical CT decreases for a larger pitch. We set $s_{\text{axial}}\approx 40$ and $s_{\text{helical}}\approx 24$ for the proposed algorithm in our experiments.

\section{Experimental results} \label{sec:jour-14-fxr:result}
This section reports numerical results for $3$D X-ray CT image reconstruction from real CT scans with different geometries using various OS-based algorithms, including
\begin{itemize}
	\item \mbox{\textbf{OS-SQS-$\mbi{M}$}}: the standard OS algorithm \cite{erdogan:99:osa} with $M$ subsets,
	\item \mbox{\textbf{OS-Nes05-$\mbi{M}$}}: the OS+momentum algorithm \cite{kim:13:axr} based on Nesterov's fast gradient method \cite{nesterov:05:smo} with $M$ subsets,
	\item \mbox{\textbf{OS-rNes05-$\mbi{M}$-$\msb{\gamma}$}}: the relaxed OS+momentum algorithm \cite{kim:13:osa} based on Nesterov's fast gradient method \cite{nesterov:05:smo} and Devolder's growing diagonal majorizer that uses an iteration-dependent diagonal Hessian $\mb{D}+\left(j+2\right)\gamma\,\mb{I}$ \cite{devolder:11:sfo} at the $j$th inner iteration with $M$ subsets,
	\item \mbox{\textbf{OS-LALM-$\mbi{M}$-$\msb{\rho}$-$\mbi{n}$}}: the proposed algorithm using a fixed AL penalty parameter $\rho$ with $M$ subsets and $n$ FISTA iterations for solving the inner constrained denoising problem, and
	\item \mbox{\textbf{OS-LALM-$\mbi{M}$-c-$\mbi{n}$}}: the proposed algorithm using the deterministic downward continuation approach described in Section~\ref{subsec:jour-14-fxr:down_cont} with $M$ subsets and $n$ FISTA iterations for solving the inner constrained denoising problem.
\end{itemize}
\mbox{OS-SQS} is a standard iterative method for tomographic reconstruction. \mbox{OS-Nes05} is a state-of-the-art method for fast X-ray CT image reconstruction using Nesterov's momentum technique, and \mbox{OS-rNes05} is the relaxed variant of \mbox{OS-Nes05} for suppressing OS artifacts using a growing diagonal majorizer. Unlike other OS-based algorithms, our proposed algorithm has additional overhead due to the iterative inner updates. However, when $n=1$, i.e., with a single gradient descent for the constrained denoising problem, all algorithms listed above have the same computational complexity (one forward/back-projection pair and $M$ regularizer gradient evaluations per iteration). Therefore, comparing the convergence rate as a function of iteration is fair. We measured the convergence rate using the RMS difference (in the region of interest) between the reconstructed image $\iter{\mb{x}}{k}$ and the almost converged reference reconstruction $\mb{x}^{\star}$ that is generated by running several iterations of the standard OS+momentum algorithm with a small $M$, followed by $2000$ iterations of a convergent (i.e., one-subset) FISTA with adaptive restart \cite{odonoghue:13:arf}.
\subsection{Shoulder scan} \label{subsec:jour-14-fxr:shoulder}
In this experiment, we reconstructed a $512\times 512\times 109$ image from a shoulder region helical CT scan, where the sinogram has size $888\times 32\times 7146$ and pitch $0.5$. The maximum number of subsets suggested by \eqref{eq:jour-14-fxr:max_M_helical} is about $40$. Figure~\ref{fig:jour-14-fxr:shoulder_ini_ref_pro} shows the cropped images from the central transaxial plane of the initial FBP image, the reference reconstruction, and the reconstructed image using the proposed algorithm (\mbox{OS-LALM-$40$-c-$1$}) at the $30$th iteration (i.e., after $30$ forward/back-projection pairs). As can be seen in Figure~\ref{fig:jour-14-fxr:shoulder_ini_ref_pro}, the reconstructed image using the proposed algorithm looks almost the same as the reference reconstruction in the display window from $800$ to $1200$ Hounsfield unit (HU). The reconstructed image using the standard OS+momentum algorithm (not shown here) also looks quite similar to the reference reconstruction.

To see the difference between the standard OS+momentum algorithm and our proposed algorithm, Figure~\ref{fig:jour-14-fxr:shoulder_diff_image} shows the difference images, i.e., $\iter{\mb{x}}{30}-\mb{x}^{\star}$, for different OS-based algorithms. We can see that the standard OS algorithm (with both $20$ and $40$ subsets) exhibits visible streak artifacts and structured high frequency noise in the difference image. When $M=20$, the difference images look similar for the standard OS+momentum algorithm and our proposed algorithm, although that of the standard OS+momentum algorithm is slightly structured and non-uniform. When $M=40$, the difference image for our proposed algorithm remains uniform, whereas some noise-like OS artifacts appear in the standard OS+momentum algorithm's difference image. The OS artifacts in the reconstructed image using the standard OS+momentum algorithm become worse when $M$ increases, e.g., $M=80$. This shows the better gradient error tolerance of our proposed algorithm when OS is used, probably due to the way we compute the search direction. Additional experimental results (of a truncated abdomen scan) that demonstrate how different OS-based algorithms behave when the number of subsets exceeds the suggested maximum number of subsets can be found in the supplementary material.

In \eqref{eq:jour-14-fxr:lalm_iterates}, the search direction $\mb{s}$ is a linear average of the current gradient and the split gradient of $\ell$. Specifically, \eqref{eq:jour-14-fxr:lalm_iterates} computes the search direction using a low-pass infinite-impulse-response (IIR) filter (across iterations), and therefore, the gradient error might be suppressed by the low-pass filter, leading to a more stable reconstruction. A similar averaging technique (with a low-pass finite-impulse-response or FIR filter) is also used in the stochastic average gradient (SAG) method \cite{leroux:13:asg,schmidt:13:mfs} for acceleration and stabilization. In comparison, the standard OS+momentum algorithm computes the search direction using only the current gradient (of the auxiliary image), so the gradient error accumulates when OS is used, providing a less stable reconstruction.

Figure~\ref{fig:jour-14-fxr:shoulder_rmsd} shows the convergence rate curves (RMS differences between the reconstructed image $\iter{\mb{x}}{k}$ and the reference reconstruction $\mb{x}^{\star}$ as a function of iteration) using OS-based algorithms with (a) $20$ subsets and (b) $40$ subsets, respectively. By exploiting the linearized AL method, the proposed algorithm accelerates the standard OS algorithm remarkably. As mentioned in Section~\ref{subsec:jour-14-fxr:os_lalm}, a smaller $\rho$ can provide greater acceleration due to the increased step size. We can see the approximate $5$, $10$, and $20$ times acceleration (comparing to the standard OS algorithm, i.e., $\rho=1$) using $\rho=0.2$, $0.1$, and $0.05$ in both figures. Note that too large step sizes can cause overshoots in early iterations. For example, the proposed algorithm with $\rho=0.05$ shows slower convergence rate in first few iterations but decreases more rapidly later. This trade-off can be overcome by using our proposed deterministic downward continuation approach. As can be seen in Figure~\ref{fig:jour-14-fxr:shoulder_rmsd}, the proposed algorithm using deterministic downward continuation reaches the lowest RMS differences (lower than $1$ HU) within only $30$ iterations! Furthermore, the slightly higher RMS difference of the standard OS+momentum algorithm with $40$ subsets shows evidence of OS artifacts.

Figure~\ref{fig:jour-14-fxr:shoulder_rmsd_n} demonstrates the effectiveness of solving the inner constrained denoising problem using FISTA (for X-ray CT image reconstruction) mentioned in Section~\ref{subsec:jour-14-fxr:os_lalm_ct}. As can be seen in Figure~\ref{fig:jour-14-fxr:shoulder_rmsd_n}, the convergence rate improves only slightly when more FISTA iterations are used for solving the inner constrained denoising problem. In practice, one FISTA iteration, i.e., $n=1$, suffices for fast and ``convergent'' X-ray CT image reconstruction. When the inner constrained denoising problem is more difficult to solve, one might want to introduce an additional split variable for the regularizer as in \cite{ramani:12:asb} at the cost of higher memory burden, thus leading to a ``high-memory'' version of OS-LALM.

\begin{figure*}
	\centering
	\includegraphics[width=\textwidth]{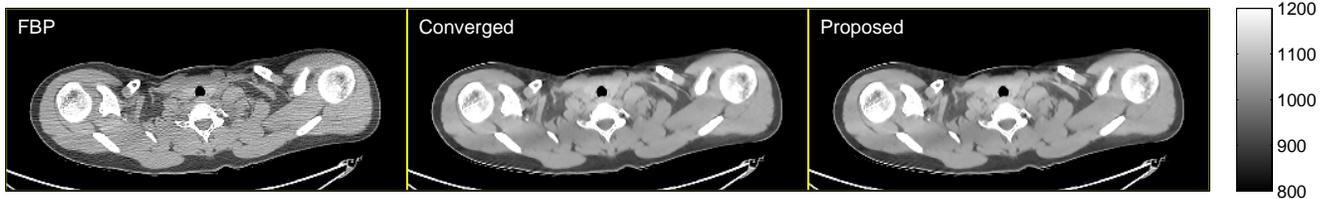}
	\caption{Shoulder scan: cropped images (displayed from $800$ to $1200$ HU) from the central transaxial plane of the initial FBP image $\iter{\mb{x}}{0}$ (left), the reference reconstruction $\mb{x}^{\star}$ (center), and the reconstructed image using the proposed algorithm (\mbox{OS-LALM-$40$-c-$1$}) at the $30$th iteration $\iter{\mb{x}}{30}$ (right).}
	\label{fig:jour-14-fxr:shoulder_ini_ref_pro}
\end{figure*}

\begin{figure*}
	\centering
	\includegraphics[width=\textwidth]{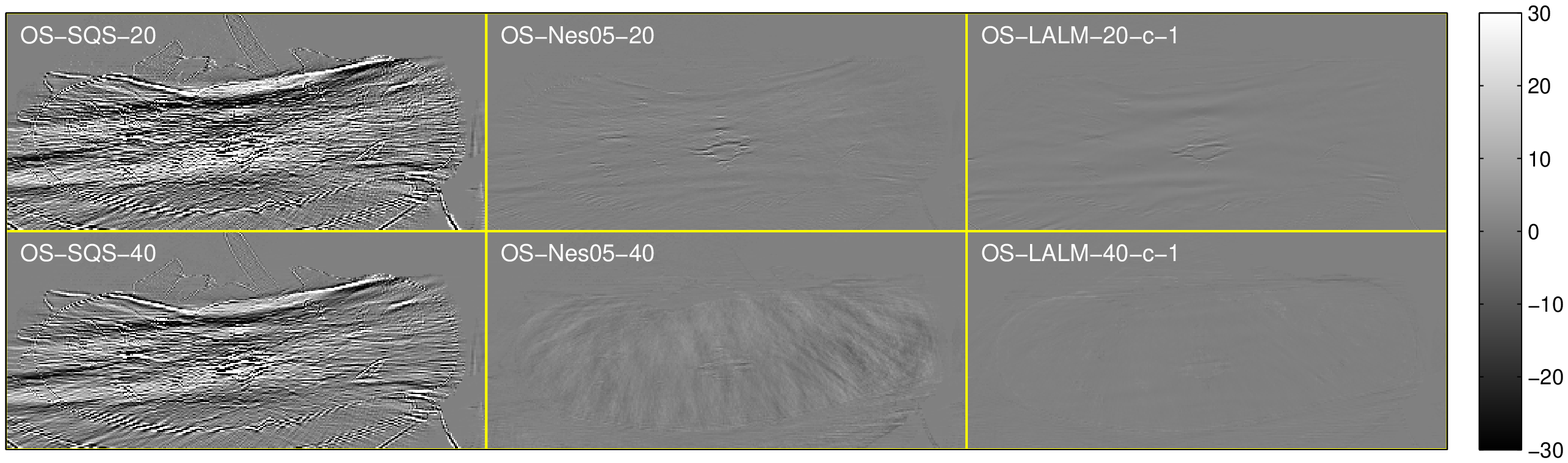}
	\caption{Shoulder scan: cropped difference images (displayed from ${-30}$ to $30$ HU) from the central transaxial plane of $\iter{\mb{x}}{30}-\mb{x}^{\star}$ using OS-based algorithms.}
	\label{fig:jour-14-fxr:shoulder_diff_image}
\end{figure*}

\begin{figure*}
	\centering
	\subfigure[]{
	\includegraphics[width=0.45\textwidth]{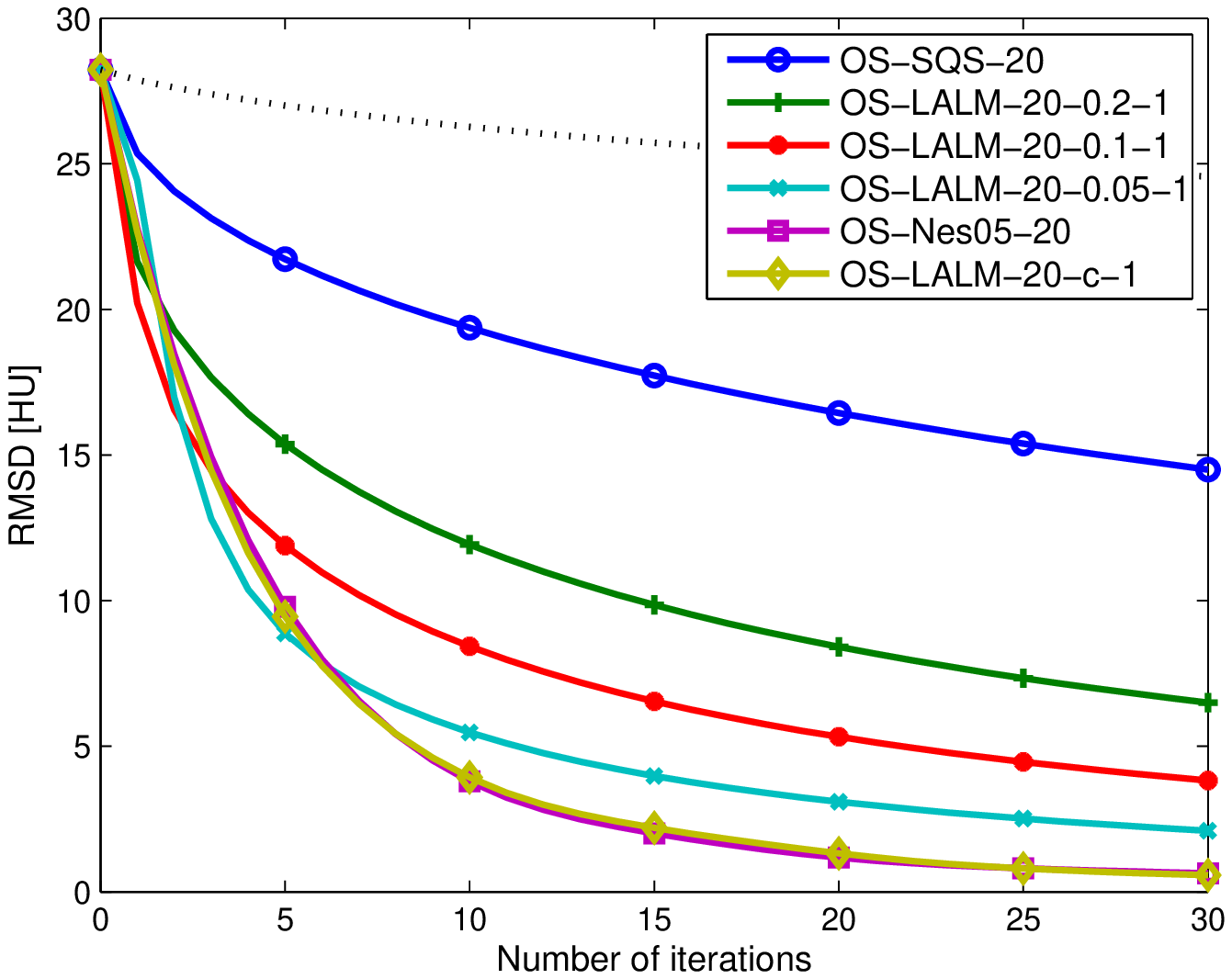}
	\label{fig:jour-14-fxr:shoulder_rmsd_20}
	}
	\subfigure[]{
	\includegraphics[width=0.45\textwidth]{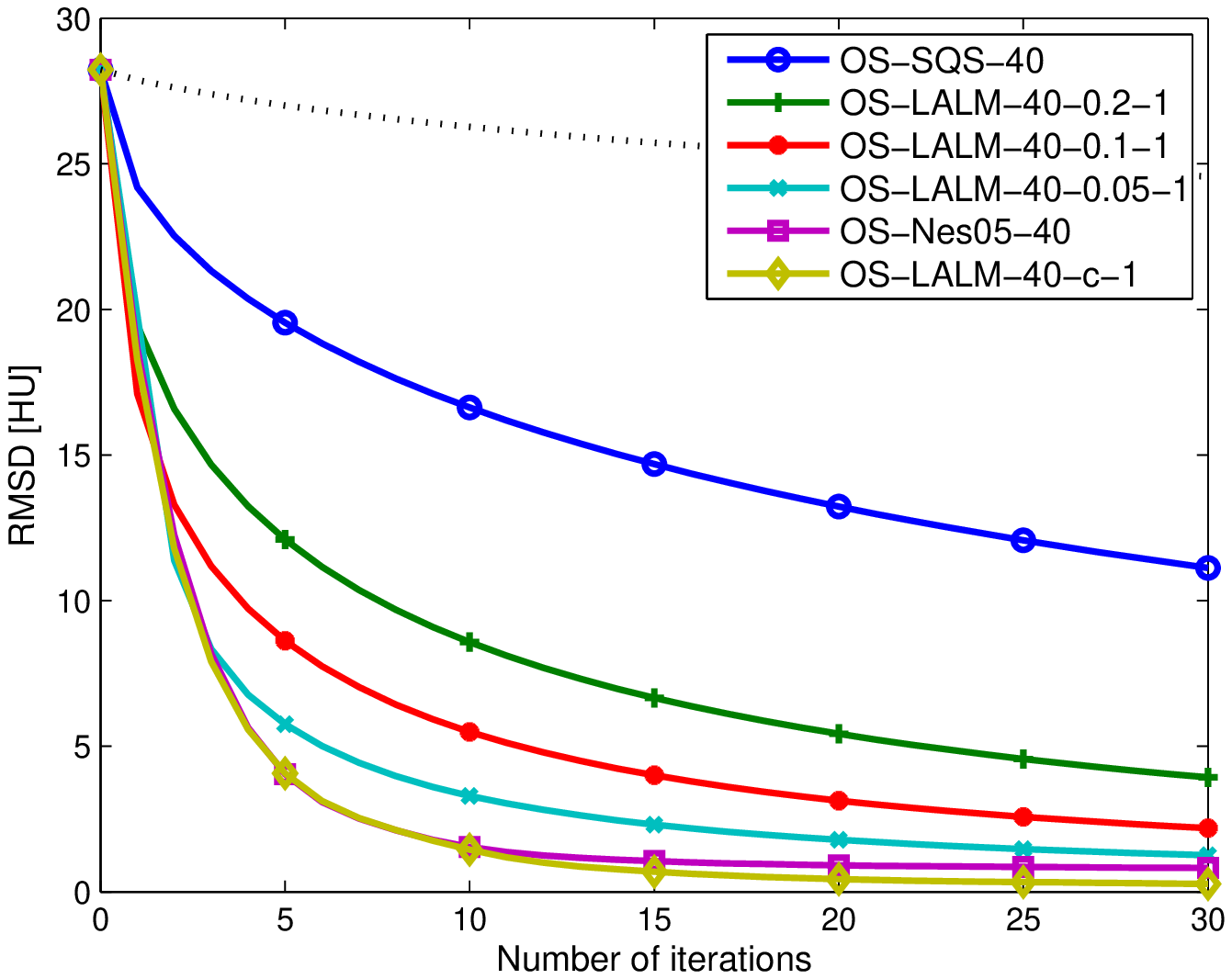}
	\label{fig:jour-14-fxr:shoulder_rmsd_40}
	}
	\caption{Shoulder scan: RMS differences between the reconstructed image $\iter{\mb{x}}{k}$ and the reference reconstruction $\mb{x}^{\star}$ as a function of iteration using OS-based algorithms with (a) $20$ subsets and (b) $40$ subsets, respectively. The dotted lines show the RMS differences using the standard OS algorithm with one subset as the baseline convergence rate.}
	\label{fig:jour-14-fxr:shoulder_rmsd}
\end{figure*}

\begin{figure}
	\centering
	\includegraphics[width=0.45\textwidth]{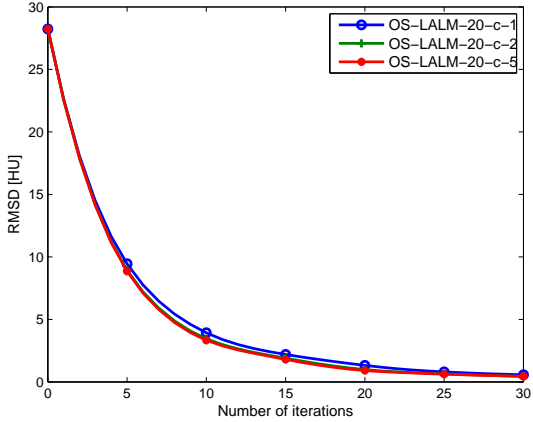}
	\caption{Shoulder scan: RMS differences between the reconstructed image $\iter{\mb{x}}{k}$ and the reference reconstruction $\mb{x}^{\star}$ as a function of iteration using the proposed algorithm with different number of FISTA iterations $n$ ($1$, $2$, and $5$) for solving the inner constrained denoising problem.}
	\label{fig:jour-14-fxr:shoulder_rmsd_n}
\end{figure}


\subsection{GE performance phantom} \label{subsec:jour-14-fxr:gepp}
In this experiment, we reconstructed a $1024\times 1024\times 90$ image from the GE performance phantom (GEPP) axial CT scan, where the sinogram has size $888\times 64\times 984$. The maximum number of subsets suggested by \eqref{eq:jour-14-fxr:max_M_axial} is about $24$. Figure~\ref{fig:jour-14-fxr:gepp_ini_ref_pro} shows the cropped images from the central transaxial plane of the initial FBP image, the reference reconstruction, and the reconstructed image using the proposed algorithm (\mbox{OS-LALM-$24$-c-$1$}) at the $30$th iteration. Again, the reconstructed image using the proposed algorithm at the $30$th iteration is very similar to the reference reconstruction.

The goal of this experiment is to evaluate the gradient error tolerance of our proposed algorithm and the recently proposed relaxed OS+momentum algorithm \cite{kim:13:osa} that trades reconstruction stability with speed by introducing relaxed momentum (i.e., a growing diagonal majorizer). We vary $\gamma$ to investigate different amounts of relaxation. When $\gamma=0$, the relaxed OS+momentum algorithm reverts to the standard OS+momentum algorithm. A larger $\gamma$ can lead to a more stable reconstruction but slower convergence. Figure~\ref{fig:jour-14-fxr:gepp_diff_image} and Figure~\ref{fig:jour-14-fxr:gepp_rmsd} show the difference images and convergence rate curves using these OS-based algorithms, respectively. As can be seen in Figure~\ref{fig:jour-14-fxr:gepp_diff_image} and \ref{fig:jour-14-fxr:gepp_rmsd}, the standard OS+momentum algorithm has even more OS artifacts than the standard OS algorithm, probably because $24$ subsets in axial CT is too aggressive for the standard OS+momentum algorithm, and we can see clear OS artifacts in the difference image and large limit cycle in the convergence rate curve. The OS artifacts are less visible as $\gamma$ increases. The case $\gamma=0.005$ achieves the best trade-off between OS artifact removal and fast convergence rate. When $\gamma$ is even larger, the relaxed OS+momentum algorithm is significantly slowed down although the difference image looks quite uniform (with some structured high frequency noise). The proposed \mbox{OS-LALM} algorithm avoids the need for such parameter tuning; one only needs to choose the number of subsets $M$. Furthermore, even for $\gamma=0.005$, the relaxed OS+momentum algorithm still has more visible OS artifacts and slower convergence rate comparing to our proposed algorithm.

\begin{figure*}
	\centering
	\includegraphics[width=\textwidth]{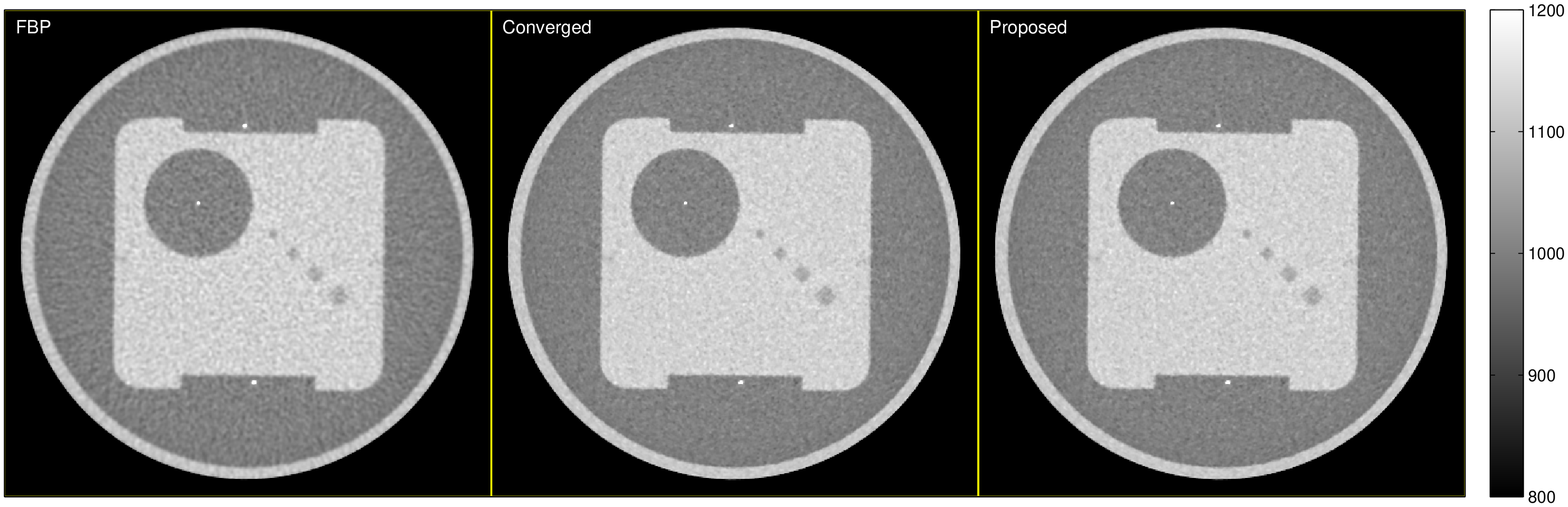}
	\caption{GE performance phantom: cropped images (displayed from $800$ to $1200$ HU) from the central transaxial plane of the initial FBP image $\iter{\mb{x}}{0}$ (left), the reference reconstruction $\mb{x}^{\star}$ (center), and the reconstructed image using the proposed algorithm (\mbox{OS-LALM-$24$-c-$1$}) at the $30$th iteration $\iter{\mb{x}}{30}$ (right).}
	\label{fig:jour-14-fxr:gepp_ini_ref_pro}
\end{figure*}

\begin{figure*}
	\centering
	\includegraphics[width=\textwidth]{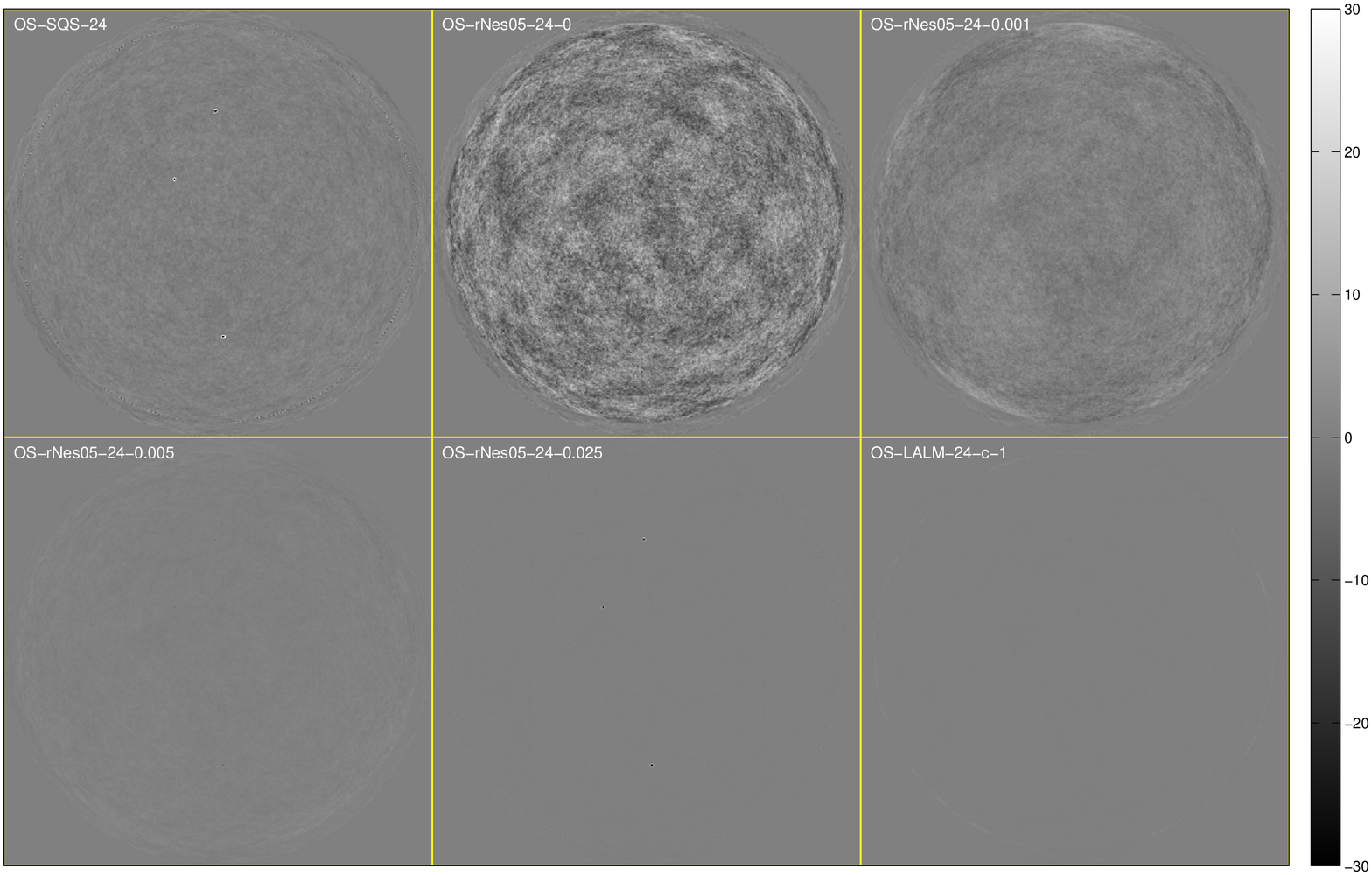}
	\caption{GE performance phantom: cropped difference images (displayed from ${-30}$ to $30$ HU) from the central transaxial plane of $\iter{\mb{x}}{30}-\mb{x}^{\star}$ using the relaxed OS+momentum algorithm \cite{kim:13:osa} and the proposed algorithm.}
	\label{fig:jour-14-fxr:gepp_diff_image}
\end{figure*}

\begin{figure}
	\centering
	\includegraphics[width=0.45\textwidth]{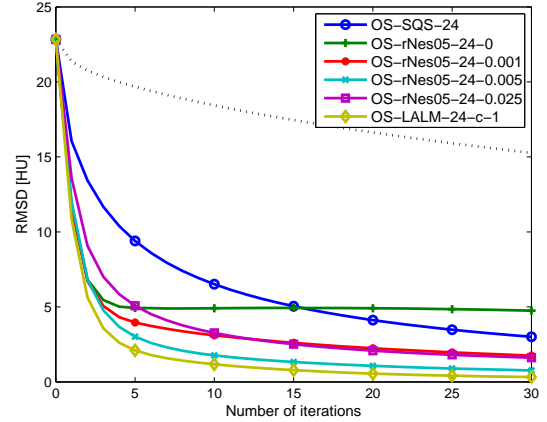}
	\caption{GE performance phantom: RMS differences between the reconstructed image $\iter{\mb{x}}{k}$ and the reference reconstruction $\mb{x}^{\star}$ as a function of iteration using the relaxed OS+momentum algorithm \cite{kim:13:osa} and the proposed algorithm with $24$ subsets. The dotted line shows the RMS differences using the standard OS algorithm with one subset as the baseline convergence rate.}
	\label{fig:jour-14-fxr:gepp_rmsd}
\end{figure}

\section{Conclusion} \label{sec:jour-14-fxr:conclusion}
The augmented Lagrangian (AL) method and ordered subsets (OS) are two powerful techniques for accelerating optimization algorithms using decomposition and approximation, respectively. This paper combined these two techniques by considering a linearized variant of the AL method and proposed a fast OS-accelerable splitting-based algorithm, \mbox{OS-LALM}, for solving regularized (weighted) least-squares problems, together with a novel deterministic downward continuation approach based on a second-order damping system. We applied our proposed algorithm to X-ray computed tomography (CT) image reconstruction problems and compared with some state-of-the-art methods using real CT scans with different geometries. Experimental results showed that our proposed algorithm exhibits fast convergence rate and excellent gradient error tolerance when OS is used. As future works, we are interested in the convergence rate analysis of the proposed algorithm with the deterministic downward continuation approach and a more rigorous convergence analysis of the proposed algorithm for $M>1$.

\section*{Acknowledgements}
The authors thank GE Healthcare for providing sinogram data in our experiments.

\bibliographystyle{ieeetr}
\bibliography{master}

\end{document}


\title{Fast X-Ray CT Image Reconstruction Using the Linearized Augmented Lagrangian Method with Ordered Subsets: Supplementary Material}
\author{Hung Nien, \textit{Student Member, IEEE}, and Jeffrey A. Fessler, \textit{Fellow, IEEE}\\[5pt]Department of Electrical Engineering and Computer Science\\University of Michigan, Ann Arbor, MI}

\maketitle

In this supplementary material, we provide the detailed convergence analysis of the linearized augmented Lagrangian (AL) method with inexact updates proposed in \cite{nien:14:fxr} together with additional experimental results.

\section{Convergence analysis of the inexact linearized AL method} \label{sec:jour-14-fxr-supp:conv_ana_inexact_lalm}
Consider a general composite convex optimization problem:
\begin{equation} \label{eq:jour-14-fxr-supp:comp_conv_opt_prob}
	\hat{\mb{x}}
	\in
	\bargmin{\mb{x}}{\fx{g}{\mb{Ax}}+\fx{h}{\mb{x}}}
\end{equation}
and its equivalent constrained minimization problem:
\begin{equation} \label{eq:jour-14-fxr-supp:eq_comp_conv_opt_prob}
	\left(\hat{\mb{x}},\hat{\mb{u}}\right)
	\in
	\bargmin{\mb{x},\mb{u}}{\fx{g}{\mb{u}}+\fx{h}{\mb{x}}}
	\text{ s.t. }
	\mb{u}=\mb{Ax} \, ,
\end{equation}
where both $g$ and $h$ are closed and proper convex functions. The inexact linearized AL methods that solve \eqref{eq:jour-14-fxr-supp:eq_comp_conv_opt_prob} are as follows:
\begin{equation} \label{eq:jour-14-fxr-supp:inexact_lalm_type_1}
	\begin{cases}
	\Bnorm{
	\iter{\mb{x}}{k+1}
	-
	\nbargmin
	{\mb{x}}
	{\fx{\phi_k}{\mb{x}}}
	}{}
	\leq
	\delta_k \\
	\iter{\mb{u}}{k+1}
	\in
	\argmin
	{\mb{u}}
	{\fx{g}{\mb{u}}+\ts\frac{\rho}{2}\norm{\mb{A}\iter{\mb{x}}{k+1}-\mb{u}-\iter{\mb{d}}{k}}{2}^2} \\
	\iter{\mb{d}}{k+1}
	=
	\iter{\mb{d}}{k}-\mb{A}\iter{\mb{x}}{k+1}+\iter{\mb{u}}{k+1} \, ,
	\end{cases}
\end{equation}
and
\begin{equation} \label{eq:jour-14-fxr-supp:inexact_lalm_type_2}
	\begin{cases}
	\Big|
	\bfx{\phi_k}{\iter{\mb{x}}{k+1}}
	-
	\nbvarmin
	{\mb{x}}
	{\fx{\phi_k}{\mb{x}}}
	\Big|
	\leq
	\varepsilon_k \\
	\iter{\mb{u}}{k+1}
	\in
	\argmin
	{\mb{u}}
	{\fx{g}{\mb{u}}+\ts\frac{\rho}{2}\norm{\mb{A}\iter{\mb{x}}{k+1}-\mb{u}-\iter{\mb{d}}{k}}{2}^2} \\
	\iter{\mb{d}}{k+1}
	=
	\iter{\mb{d}}{k}-\mb{A}\iter{\mb{x}}{k+1}+\iter{\mb{u}}{k+1} \, ,
	\end{cases}
\end{equation}
where
\begin{equation} \label{eq:jour-14-fxr-supp:def_phi_k}
	\fx{\phi_k}{\mb{x}}
	\teq
	\fx{h}{\mb{x}}+\bfx{\breve{\theta}_k}{\mb{x};\iter{\mb{x}}{k}} \, ,
\end{equation}
and
\begin{equation} \label{eq:jour-14-fxr-supp:sqs_quad_al_penalty}
	\bfx{\breve{\theta}_k}{\mb{x};\iter{\mb{x}}{k}}
	\teq
	\bfx{\theta_k}{\iter{\mb{x}}{k}}
	+
	\biprod{\nabla\bfx{\theta_k}{\iter{\mb{x}}{k}}}{\mb{x}-\iter{\mb{x}}{k}}
	+
	\ts\frac{\rho L}{2}\norm{\mb{x}-\iter{\mb{x}}{k}}{2}^2
\end{equation}
is the separable quadratic surrogate (SQS) function of
\begin{equation} \label{eq:jour-14-fxr-supp:quad_al_penalty}
	\fx{\theta_k}{\mb{x}}\teq\ts\frac{\rho}{2}\norm{\mb{Ax}-\iter{\mb{u}}{k}-\iter{\mb{d}}{k}}{2}^2
\end{equation}
with $L>\norm{\mb{A}}{2}^2=\fx{\lambda_{\text{max}}}{\mb{A}'\mb{A}}$, $\left\{\delta_k\right\}_{k=0}^{\infty}$ and $\left\{\varepsilon_k\right\}_{k=0}^{\infty}$ are two non-negative sequences, $\mb{d}$ is the scaled Lagrange multiplier of the split variable $\mb{u}$, and $\rho>0$ is the corresponding AL penalty parameter. Furthermore, in \cite{nien:14:fxr}, we also showed that the inexact linearized AL method is equivalent to the invexact version of the Chambolle-Pock first-order primal-dual algorithm (CPPDA) \cite{chambolle:11:afo}:
\begin{equation} \label{eq:jour-14-fxr-supp:general_cppd_iterates}
	\begin{cases}
	\iter{\mb{x}}{k+1}
	\in
	\prox{\sigma h}{\iter{\mb{x}}{k}-\sigma\mb{A}'\iter{\bar{\mb{z}}}{k}} \\
	\iter{\mb{z}}{k+1}
	\in
	\prox{\tau g^*}{\iter{\mb{z}}{k}+\tau\mb{A}\iter{\mb{x}}{k+1}} \\
	\iter{\bar{\mb{z}}}{k+1}
	=
	\iter{\mb{z}}{k+1}+\left(\iter{\mb{z}}{k+1}-\iter{\mb{z}}{k}\right)
	\end{cases}
\end{equation}
that solves the minimax problem:
\begin{equation} \label{eq:jour-14-fxr-supp:eq_comp_conv_opt_prob_minimax}
	\left(\hat{\mb{z}},\hat{\mb{x}}\right)
	\in
	\argminmax
	{\mb{z}}{\mb{x}}
	{\fx{\Omega}{\mb{z},\mb{x}}\teq\iprod{-\mb{A}'\mb{z}}{\mb{x}}+\fx{g^*}{\mb{z}}-\fx{h}{\mb{x}}}
\end{equation}
with $\mb{z}=-\tau\mb{d}$, $\sigma=\rho^{-1}t$, $\tau=\rho$, and $t\teq 1/L$, where $\text{prox}_f$ denotes the proximal mapping of $f$ defined as:
\begin{equation} \label{eq:jour-14-fxr-supp:def_prox}
	\prox{f}{\mb{z}}
	\teq
	\argmin{\mb{x}}{\fx{f}{\mb{x}}+\ts\frac{1}{2}\norm{\mb{x}-\mb{z}}{2}^2} \, ,
\end{equation}
and $f^*$ denotes the convex conjugate of a function $f$. Note that $g^{**}=g$ and $h^{**}=h$ since both $g$ and $h$ are closed, proper, and convex.

\subsection{Proof of Theorem 1} \label{subsec:jour-14-fxr-supp:proof_thm_1}
\begin{theorem} \label{thm:jour-14-fxr-supp:inexact_lalm_type_1}
Consider a constrained composite convex optimization problem \eqref{eq:jour-14-fxr-supp:eq_comp_conv_opt_prob} where both $g$ and $h$ are closed and proper convex functions. Let $\rho>0$ and $\left\{\delta_k\right\}_{k=0}^{\infty}$ denote a non-negative sequence such that
\begin{equation} \label{eq:jour-14-fxr-supp:error_bound_type_1}
	\sum_{k=0}^{\infty}\delta_k<\infty \, .
\end{equation}
If \eqref{eq:jour-14-fxr-supp:eq_comp_conv_opt_prob} has a solution $\left(\hat{\mb{x}},\hat{\mb{u}}\right)$, then the sequence of updates $\left\{\left(\iter{\mb{x}}{k},\iter{\mb{u}}{k}\right)\right\}_{k=0}^{\infty}$ generated by the inexact linearized AL method \eqref{eq:jour-14-fxr-supp:inexact_lalm_type_1} converges to $\left(\hat{\mb{x}},\hat{\mb{u}}\right)$; otherwise, at least one of the sequences $\left\{\left(\iter{\mb{x}}{k},\iter{\mb{u}}{k}\right)\right\}_{k=0}^{\infty}$ or $\left\{\iter{\mb{d}}{k}\right\}_{k=0}^{\infty}$ diverges.
\end{theorem}
\begin{proof} \label{pf:jour-14-fxr-supp:proof_inexact_lalm_type_1}
To prove this theorem, we first consider the exact linearized AL method:
\begin{equation} \label{eq:jour-14-fxr-supp:exact_lalm}
	\begin{cases}
	\iter{\mb{x}}{k+1}
	\in
	\argmin
	{\mb{x}}
	{\fx{h}{\mb{x}}+\bfx{\breve{\theta}_k}{\mb{x};\iter{\mb{x}}{k}}} \\
	\iter{\mb{u}}{k+1}
	\in
	\argmin
	{\mb{u}}
	{\fx{g}{\mb{u}}+\ts\frac{\rho}{2}\norm{\mb{A}\iter{\mb{x}}{k+1}-\mb{u}-\iter{\mb{d}}{k}}{2}^2} \\
	\iter{\mb{d}}{k+1}
	=
	\iter{\mb{d}}{k}-\mb{A}\iter{\mb{x}}{k+1}+\iter{\mb{u}}{k+1} \, .
	\end{cases}
\end{equation}
Note that
\begin{align} \label{eq:jour-13-fxr:sqs_al_penalty_proximal_form}
	\bfx{\breve{\theta}_k}{\mb{x};\iter{\mb{x}}{k}}
	&=
	\bfx{\theta_k}{\iter{\mb{x}}{k}}
	+
	\biprod{\nabla\bfx{\theta_k}{\iter{\mb{x}}{k}}}{\mb{x}-\iter{\mb{x}}{k}}
	+
	\ts\frac{\rho L}{2}\norm{\mb{x}-\iter{\mb{x}}{k}}{2}^2 \nonumber \\
	&=
	\bfx{\theta_k}{\iter{\mb{x}}{k}}
	+
	\biprod{\nabla\bfx{\theta_k}{\iter{\mb{x}}{k}}}{\mb{x}-\iter{\mb{x}}{k}}
	+
	\ts\frac{\rho}{2}\norm{\mb{x}-\iter{\mb{x}}{k}}{\mb{A}'\mb{A}}^2
	+
	\ts\frac{\rho}{2}\norm{\mb{x}-\iter{\mb{x}}{k}}{L\mb{I}-\mb{A}'\mb{A}}^2 \nonumber \\
	&=
	\fx{\theta_k}{\mb{x}}
	+
	\ts\frac{\rho}{2}\norm{\mb{x}-\iter{\mb{x}}{k}}{\mb{G}}^2 \, ,
\end{align}
where $\mb{G}\teq L\mb{I}-\mb{A}'\mb{A}\succ 0$. Therefore, the exact linearized AL method can also be written as
\begin{equation} \label{eq:jour-14-fxr-supp:exact_lalm_prox_pt}
	\begin{cases}
	\iter{\mb{x}}{k+1}
	\in
	\argmin
	{\mb{x}}
	{\fx{h}{\mb{x}}+\ts\frac{\rho}{2}\norm{\mb{Ax}-\iter{\mb{u}}{k}-\iter{\mb{d}}{k}}{2}^2+\ts\frac{\rho}{2}\norm{\mb{x}-\iter{\mb{x}}{k}}{\mb{G}}^2} \\
	\iter{\mb{u}}{k+1}
	\in
	\argmin
	{\mb{u}}
	{\fx{g}{\mb{u}}+\ts\frac{\rho}{2}\norm{\mb{A}\iter{\mb{x}}{k+1}-\mb{u}-\iter{\mb{d}}{k}}{2}^2} \\
	\iter{\mb{d}}{k+1}
	=
	\iter{\mb{d}}{k}-\mb{A}\iter{\mb{x}}{k+1}+\iter{\mb{u}}{k+1} \, .
	\end{cases}
\end{equation}
Now, consider another constrained minimization problem that is also equivalent to \eqref{eq:jour-14-fxr-supp:comp_conv_opt_prob} but uses two split variables:
\begin{equation} \label{eq:jour-14-fxr-supp:eq_comp_conv_opt_prob_with_two_splits}
	\left(\hat{\mb{x}},\hat{\mb{u}},\hat{\mb{v}}\right)
	\in
	\bargmin{\mb{x},\mb{u},\mb{v}}{\fx{g}{\mb{u}}+\fx{h}{\mb{x}}}
	\text{ s.t. }
	\begin{bmatrix}
	\mb{u} \\ \mb{v}
	\end{bmatrix}
	=
	\underbrace{
	\begin{bmatrix}
	\mb{A} \\ \mb{G}^{1/2}
	\end{bmatrix}}_{\mb{S}}
	\mb{x} \, .
\end{equation}
The corresponding augmented Lagrangian and ADMM iterates \cite{afonso:11:aal} are
\begin{equation} \label{eq:jour-14-fxr-supp:aug_lagrangian_with_two_splits}
	\fx{\LA}{\mb{x},\mb{u},\mb{d},\mb{v},\mb{e};\rho,\eta}
	\teq
	\fx{g}{\mb{u}}+\fx{h}{\mb{x}}+\ts\frac{\rho}{2}\norm{\mb{Ax}-\mb{u}-\mb{d}}{2}^2+\ts\frac{\eta}{2}\norm{\mb{G}^{1/2}\mb{x}-\mb{v}-\mb{e}}{2}^2
\end{equation}
and
\begin{equation} \label{eq:jour-14-fxr-supp:exact_admm_with_two_splits}
	\begin{cases}
	\iter{\mb{x}}{k+1}
	\in
	\argmin
	{\mb{x}}
	{\fx{h}{\mb{x}}+\ts\frac{\rho}{2}\norm{\mb{Ax}-\iter{\mb{u}}{k}-\iter{\mb{d}}{k}}{2}^2+\ts\frac{\eta}{2}\norm{\mb{G}^{1/2}\mb{x}-\iter{\mb{v}}{k}-\iter{\mb{e}}{k}}{2}^2} \\
	\iter{\mb{u}}{k+1}
	\in
	\argmin
	{\mb{u}}
	{\fx{g}{\mb{u}}+\ts\frac{\rho}{2}\norm{\mb{A}\iter{\mb{x}}{k+1}-\mb{u}-\iter{\mb{d}}{k}}{2}^2} \\
	\iter{\mb{d}}{k+1}
	=
	\iter{\mb{d}}{k}-\mb{A}\iter{\mb{x}}{k+1}+\iter{\mb{u}}{k+1} \\
	\iter{\mb{v}}{k+1}
	=
	\mb{G}^{1/2}\iter{\mb{x}}{k+1}-\iter{\mb{e}}{k} \\
	\iter{\mb{e}}{k+1}
	=
	\iter{\mb{e}}{k}-\mb{G}^{1/2}\iter{\mb{x}}{k+1}+\iter{\mb{v}}{k+1} \, ,
	\end{cases}
\end{equation}
where $\mb{e}$ is the scaled Lagrange multiplier of the split variable $\mb{v}$, and $\eta>0$ is the corresponding AL penalty parameter. Note that since $\mb{G}$ is positive definite, $\mb{S}$ defined in \eqref{eq:jour-14-fxr-supp:eq_comp_conv_opt_prob_with_two_splits} has full column rank. Hence, the ADMM iterates \eqref{eq:jour-14-fxr-supp:exact_admm_with_two_splits} are convergent \mbox{\cite[Theorem 8]{eckstein:92:otd}}. Solving the last two iterates in \eqref{eq:jour-14-fxr-supp:exact_admm_with_two_splits} yields identities
\begin{equation} \label{eq:jour-14-fxr-supp:al_identity_of_redundant_split}
	\begin{cases}
	\iter{\mb{v}}{k+1}=\mb{G}^{1/2}\iter{\mb{x}}{k+1} \\
	\iter{\mb{e}}{k+1}=\mb{0}
	\end{cases}
\end{equation}
if we initialize $\mb{e}$ as $\iter{\mb{e}}{0}=\mb{0}$. Substituting \eqref{eq:jour-14-fxr-supp:al_identity_of_redundant_split} into \eqref{eq:jour-14-fxr-supp:exact_admm_with_two_splits}, we have the equivalent ADMM iterates:
\begin{equation} \label{eq:jour-14-fxr-supp:eq_exact_admm_with_two_splits}
	\begin{cases}
	\iter{\mb{x}}{k+1}
	\in
	\argmin
	{\mb{x}}
	{\fx{h}{\mb{x}}+\ts\frac{\rho}{2}\norm{\mb{Ax}-\iter{\mb{u}}{k}-\iter{\mb{d}}{k}}{2}^2+\ts\frac{\eta}{2}\norm{\mb{G}^{1/2}\mb{x}-\mb{G}^{1/2}\iter{\mb{x}}{k}}{2}^2} \\
	\iter{\mb{u}}{k+1}
	\in
	\argmin
	{\mb{u}}
	{\fx{g}{\mb{u}}+\ts\frac{\rho}{2}\norm{\mb{A}\iter{\mb{x}}{k+1}-\mb{u}-\iter{\mb{d}}{k}}{2}^2} \\
	\iter{\mb{d}}{k+1}
	=
	\iter{\mb{d}}{k}-\mb{A}\iter{\mb{x}}{k+1}+\iter{\mb{u}}{k+1} \, .
	\end{cases}
\end{equation}
When $\eta=\rho$, the equivalent ADMM iterates \eqref{eq:jour-14-fxr-supp:eq_exact_admm_with_two_splits} reduce to \eqref{eq:jour-14-fxr-supp:exact_lalm_prox_pt}. Therefore, the linearized AL method is a convergent ADMM! Finally, by using \cite[Theorem 8]{eckstein:92:otd}, the linearized AL method is convergent if the error of $\mb{x}$-update is summable. That is, the inexact linearized AL method \eqref{eq:jour-14-fxr-supp:inexact_lalm_type_1} is convergent if the non-negative sequence $\left\{\delta_k\right\}_{k=0}^{\infty}$ satisfies $\sum_{k=0}^{\infty}\delta_k < \infty$.
\end{proof}

\subsection{Proof of Theorem 2} \label{subsec:jour-14-fxr-supp:proof_thm_2}
\begin{theorem} \label{thm:jour-14-fxr-supp:inexact_lalm_type_2}
Consider a minimax problem \eqref{eq:jour-14-fxr-supp:eq_comp_conv_opt_prob_minimax} where both $g$ and $h$ are closed and proper convex functions. Suppose it has a saddle-point $\left(\hat{\mb{z}},\hat{\mb{x}}\right)$. Note that since the minimization problem \eqref{eq:jour-14-fxr-supp:comp_conv_opt_prob} happens to be the dual problem of \eqref{eq:jour-14-fxr-supp:eq_comp_conv_opt_prob_minimax}, $\hat{\mb{x}}$ is also a solution of \eqref{eq:jour-14-fxr-supp:comp_conv_opt_prob}. Let $\rho>0$ and $\left\{\varepsilon_k\right\}_{k=0}^{\infty}$ denote a non-negative sequence such that
\begin{equation} \label{eq:jour-14-fxr-supp:error_bound_type_2}
	\sum_{k=0}^{\infty}\sqrt{\varepsilon_k}<\infty \, .
\end{equation}
Then, the sequence of updates $\left\{\left({-\rho}\iter{\mb{d}}{k},\iter{\mb{x}}{k}\right)\right\}_{k=0}^{\infty}$ generated by the inexact linearized AL method \eqref{eq:jour-14-fxr-supp:inexact_lalm_type_2} is a bounded sequence that converges to $\left(\hat{\mb{z}},\hat{\mb{x}}\right)$, and the primal-dual gap of $\left(\mb{z}_k,\mb{x}_k\right)$ has the following bound:
\begin{equation} \label{eq:jour-14-fxr-supp:primal_dual_gap_bound_with_error}
	\bfx{\Omega}{\mb{z}_k,\hat{\mb{x}}}-\bfx{\Omega}{\hat{\mb{z}},\mb{x}_k}
	\leq
	\frac{\left(C+2A_k+\sqrt{B_k}\right)^2}{k} \, ,
\end{equation}
where $\mb{z}_k\teq\frac{1}{k}\sum_{j=1}^k\big({-\rho}\iter{\mb{d}}{j}\big)$, $\mb{x}_k\teq\frac{1}{k}\sum_{j=1}^k\iter{\mb{x}}{j}$,
\begin{equation} \label{eq:jour-14-fxr-supp:def_C}
	C\teq\frac{\norm{\iter{\mb{x}}{0}-\hat{\mb{x}}}{2}}{\sqrt{2\rho^{-1}t}}+\frac{\norm{\big({-\rho}\iter{\mb{d}}{0}\big)-\hat{\mb{z}}}{2}}{\sqrt{2\rho}} \, ,
\end{equation}
\begin{equation} \label{eq:jour-14-fxr-supp:def_Ak}
	A_k\teq\sum_{j=1}^k\sqrt{\frac{\varepsilon_{j-1}}{\big(1-t\norm{\mb{A}}{2}^2\big)\rho^{-1}t}} \, ,
\end{equation}
and
\begin{equation} \label{eq:jour-14-fxr-supp:def_Bk}
	B_k\teq\sum_{j=1}^k\varepsilon_{j-1} \, .
\end{equation}
\end{theorem}
\begin{proof} \label{pf:jour-14-fxr-supp:proof_inexact_lalm_type_2}
As mentioned before, the inexact linearized AL method is the inexact version of CPPDA with a specific choice of $\sigma$ and $\tau$ and a substitution $\mb{z}=-\tau\mb{d}$. Here, we just prove the convergence of the inexact CPPDA by extending the analysis in \cite{chambolle:11:afo}, and the inexact linearized AL method is simply a special case of the inexact CPPDA. However, since the proximal mapping in the $\mb{x}$-update of the inexact CPPDA is solved inexactly, the existing analysis is not applicable. To solve this problem, we adopt the error analysis technique developed in \cite{schmidt:11:cro}. We first define the inexact proximal mapping
\begin{equation} \label{eq:jour-14-fxr-supp:inexact_prox}
	\mb{u}
	\overset{\varepsilon}{\approx}
	\prox{\phi}{\mb{v}}
\end{equation}
to be the mapping that satisfies
\begin{equation} \label{eq:jour-14-fxr-supp:def_inexact_prox}
	\fx{\phi}{\mb{u}}+\ts\frac{1}{2}\norm{\mb{u}-\mb{v}}{2}^2
	\leq
	\varepsilon
	+
	\varmin
	{\bar{\mb{u}}}
	{\fx{\phi}{\bar{\mb{u}}}+\ts\frac{1}{2}\norm{\bar{\mb{u}}-\mb{v}}{2}^2} \, .
\end{equation}
Therefore, the inexact CPPDA is defined as
\begin{equation} \label{eq:jour-14-fxr-supp:inexact_cppd_type_2}
	\begin{cases}
	\iter{\mb{x}}{k+1}
	\overset{\varepsilon_k}{\approx}
	\prox{\sigma h}{\iter{\mb{x}}{k}-\sigma\mb{A}'\iter{\bar{\mb{z}}}{k}} \\
	\iter{\mb{z}}{k+1}
	\in
	\prox{\tau g^*}{\iter{\mb{z}}{k}+\tau\mb{A}\iter{\mb{x}}{k+1}} \\
	\iter{\bar{\mb{z}}}{k+1}
	=
	\iter{\mb{z}}{k+1}+\left(\iter{\mb{z}}{k+1}-\iter{\mb{z}}{k}\right)
	\end{cases}
\end{equation}
with $\sigma\tau\norm{\mb{A}}{2}^2<1$. One can verify that with $\mb{z}=-\tau\mb{d}$, $\sigma=\rho^{-1}t$, and $\tau=\rho$, the inexact CPPDA \eqref{eq:jour-14-fxr-supp:inexact_cppd_type_2} is equivalent to the inexact linearized AL method \eqref{eq:jour-14-fxr-supp:inexact_lalm_type_2}. Schmidt \textit{et al.} showed that
\begin{equation} \label{eq:jour-14-fxr-supp:opt_cond_prox_with_error}
	\mb{u}
	\overset{\varepsilon}{\approx}
	\prox{\phi}{\mb{v}}
	\Leftrightarrow
	\mb{v}-\mb{u}-\mb{f}\in\partial_{\varepsilon}\fx{\phi}{\mb{u}}
\end{equation}
with $\norm{\mb{f}}{2}\leq\sqrt{2\varepsilon}$, and for any $\mb{s}\in\partial_{\varepsilon}\fx{\phi}{\mb{u}}$,
\begin{equation} \label{eq:jour-14-fxr-supp:linear_lower_bd_prox_with_error}
	\fx{\phi}{\mb{w}}\geq\fx{\phi}{\mb{u}}+\mb{s}'\left(\mb{w}-\mb{u}\right)-\varepsilon
\end{equation}
for all $\mb{w}$, where $\partial_{\varepsilon}\fx{\phi}{\mb{u}}$ denotes the $\varepsilon$-subdifferential of $\phi$ at $\mb{u}$ \cite[Lemma 2]{schmidt:11:cro}. When $\varepsilon=0$, \eqref{eq:jour-14-fxr-supp:opt_cond_prox_with_error} and \eqref{eq:jour-14-fxr-supp:linear_lower_bd_prox_with_error} reduce to the standard optimality condition of a proximal mapping and the definition of subgradient, respectively. At the $j$th iteration, $j=0,\ldots,k-1$, the updates generated by the inexact CPPDA \eqref{eq:jour-14-fxr-supp:inexact_cppd_type_2} satisfy
\begin{equation} \label{eq:jour-14-fxr-supp:inexact_cppda_type_2_opt_cond}
	\begin{cases}
	\big(\iter{\mb{x}}{j}-\sigma\mb{A}'\iter{\bar{\mb{z}}}{j}\big)-\iter{\mb{x}}{j+1}-\iter{\mb{f}}{j}
	\in
	\partial_{\varepsilon_j}\left(\sigma h\right)\big(\iter{\mb{x}}{j+1}\big) \\
	\big(\iter{\mb{z}}{j}+\tau\mb{A}\iter{\mb{x}}{j+1}\big)-\iter{\mb{z}}{j+1}
	\in
	\partial\left(\tau g^*\right)\big(\iter{\mb{z}}{j+1}\big) \, .
	\end{cases}
\end{equation}
In other words,
\begin{equation} \label{eq:jour-14-fxr-supp:inexact_cppda_type_2_sg_x}
	\frac{\iter{\mb{x}}{j}-\iter{\mb{x}}{j+1}}{\sigma}-\mb{A}'\iter{\bar{\mb{z}}}{j}-\frac{\iter{\mb{f}}{j}}{\sigma}
	\in
	\partial_{\varepsilon_j}\bfx{h}{\iter{\mb{x}}{j+1}}
\end{equation}
and
\begin{equation} \label{eq:jour-14-fxr-supp:inexact_cppda_type_2_sg_z}
	\frac{\iter{\mb{z}}{j}-\iter{\mb{z}}{j+1}}{\tau}+\mb{A}\iter{\mb{x}}{j+1}
	\in
	\partial\bfx{g^*}{\iter{\mb{z}}{j+1}} \, ,
\end{equation}
where $\norm{\iter{\mb{f}}{j}}{2}\leq\sqrt{2\varepsilon_j}$. From \eqref{eq:jour-14-fxr-supp:inexact_cppda_type_2_sg_x}, we have
\begin{align}
	\fx{h}{\mb{x}}
	&\geq
	\bfx{h}{\iter{\mb{x}}{j+1}}+\biprod{\partial_{\varepsilon_j}\bfx{h}{\iter{\mb{x}}{j+1}}}{\mb{x}-\iter{\mb{x}}{j+1}}-\varepsilon_j \nonumber \\
	&=
	\bfx{h}{\iter{\mb{x}}{j+1}}
	+
	\biprod{\ts\frac{\iter{\mb{x}}{j}-\iter{\mb{x}}{j+1}}{\sigma}}{\mb{x}-\iter{\mb{x}}{j+1}}
	-
	\biprod{\mb{A}'\iter{\bar{\mb{z}}}{j}}{\mb{x}-\iter{\mb{x}}{j+1}}
	-
	\biprod{\ts\frac{\iter{\mb{f}}{j}}{\sigma}}{\mb{x}-\iter{\mb{x}}{j+1}}
	-
	\varepsilon_j \nonumber \\
	&=
	\bfx{h}{\iter{\mb{x}}{j+1}}
	+
	\ts\frac{1}{2\sigma}\big(\norm{\iter{\mb{x}}{j+1}-\mb{x}}{2}^2+\norm{\iter{\mb{x}}{j+1}-\iter{\mb{x}}{j}}{2}^2-\norm{\iter{\mb{x}}{j}-\mb{x}}{2}^2\big) \nonumber \\
	&\qquad
	+
	\biprod{{-\mb{A}}'\big(\iter{\bar{\mb{z}}}{j}-\iter{\mb{z}}{j+1}\big)}{\mb{x}-\iter{\mb{x}}{j+1}}
	+
	\biprod{{-\mb{A}}'\iter{\mb{z}}{j+1}}{\mb{x}-\iter{\mb{x}}{j+1}}
	-
	\biprod{\ts\frac{\iter{\mb{f}}{j}}{\sigma}}{\mb{x}-\iter{\mb{x}}{j+1}}
	-
	\varepsilon_j \nonumber \\
	&\geq
	\bfx{h}{\iter{\mb{x}}{j+1}}
	+
	\ts\frac{1}{2\sigma}\big(\norm{\iter{\mb{x}}{j+1}-\mb{x}}{2}^2+\norm{\iter{\mb{x}}{j+1}-\iter{\mb{x}}{j}}{2}^2-\norm{\iter{\mb{x}}{j}-\mb{x}}{2}^2\big) \nonumber \\
	&\qquad
	+
	\biprod{{-\mb{A}}'\big(\iter{\bar{\mb{z}}}{j}-\iter{\mb{z}}{j+1}\big)}{\mb{x}-\iter{\mb{x}}{j+1}}
	+
	\biprod{{-\mb{A}}'\iter{\mb{z}}{j+1}}{\mb{x}-\iter{\mb{x}}{j+1}}
	-
	\ts\frac{1}{\sigma}\norm{\iter{\mb{f}}{j}}{2}\norm{\mb{x}-\iter{\mb{x}}{j+1}}{2}
	-
	\varepsilon_j \nonumber \\
	&\geq
	\bfx{h}{\iter{\mb{x}}{j+1}}
	+
	\ts\frac{1}{2\sigma}\big(\norm{\iter{\mb{x}}{j+1}-\mb{x}}{2}^2+\norm{\iter{\mb{x}}{j+1}-\iter{\mb{x}}{j}}{2}^2-\norm{\iter{\mb{x}}{j}-\mb{x}}{2}^2\big) \nonumber \\
	&\qquad
	+
	\biprod{{-\mb{A}}'\big(\iter{\bar{\mb{z}}}{j}-\iter{\mb{z}}{j+1}\big)}{\mb{x}-\iter{\mb{x}}{j+1}}
	+
	\biprod{{-\mb{A}}'\iter{\mb{z}}{j+1}}{\mb{x}-\iter{\mb{x}}{j+1}}
	-
	\ts\frac{\sqrt{2\varepsilon_j}}{\sigma}\norm{\mb{x}-\iter{\mb{x}}{j+1}}{2}
	-
	\varepsilon_j \label{eq:jour-14-fxr-supp:linear_lower_bd_x}
\end{align}
for any $\mb{x}\in\text{Dom}\,h$. From \eqref{eq:jour-14-fxr-supp:inexact_cppda_type_2_sg_z}, we have
\begin{align}
	\fx{g^*}{\mb{z}}
	&\geq
	\bfx{g^*}{\iter{\mb{z}}{j+1}}+\biprod{\partial\bfx{g^*}{\iter{\mb{z}}{j+1}}}{\mb{z}-\iter{\mb{z}}{j+1}} \nonumber \\
	&=
	\bfx{g^*}{\iter{\mb{z}}{j+1}}
	+
	\biprod{\ts\frac{\iter{\mb{z}}{j}-\iter{\mb{z}}{j+1}}{\tau}}{\mb{z}-\iter{\mb{z}}{j+1}}
	+
	\biprod{\mb{A}\iter{\mb{x}}{j+1}}{\mb{z}-\iter{\mb{z}}{j+1}} \nonumber \\
	&=
	\bfx{g^*}{\iter{\mb{z}}{j+1}}
	+
	\ts\frac{1}{2\tau}\big(\norm{\iter{\mb{z}}{j+1}-\mb{z}}{2}^2+\norm{\iter{\mb{z}}{j+1}-\iter{\mb{z}}{j}}{2}^2-\norm{\iter{\mb{z}}{j}-\mb{z}}{2}^2\big)
	-
	\biprod{{-\mb{A}}'\big(\mb{z}-\iter{\mb{z}}{j+1}\big)}{\iter{\mb{x}}{j+1}} \label{eq:jour-14-fxr-supp:linear_low_bd_z}
\end{align}
for any $\mb{z}\in\text{Dom}\,g^*$. Summing \eqref{eq:jour-14-fxr-supp:linear_lower_bd_x} and \eqref{eq:jour-14-fxr-supp:linear_low_bd_z}, it follows:
\begin{align}
	&
	\frac{\norm{\iter{\mb{x}}{j}-\mb{x}}{2}^2}{2\sigma}+\frac{\norm{\iter{\mb{z}}{j}-\mb{z}}{2}^2}{2\tau}
	\geq
	\left(\bfx{\Omega}{\iter{\mb{z}}{j+1},\mb{x}}-\bfx{\Omega}{\mb{z},\iter{\mb{x}}{j+1}}\right) \nonumber \\
	&\qquad\qquad
	+
	\frac{\norm{\iter{\mb{x}}{j+1}-\mb{x}}{2}^2}{2\sigma}+\frac{\norm{\iter{\mb{z}}{j+1}-\mb{z}}{2}^2}{2\tau}
	+
	\frac{\norm{\iter{\mb{x}}{j+1}-\iter{\mb{x}}{j}}{2}^2}{2\sigma}+\frac{\norm{\iter{\mb{z}}{j+1}-\iter{\mb{z}}{j}}{2}^2}{2\tau} \nonumber \\
	&\qquad\qquad
	+
	\biprod{{-\mb{A}}'\big(\iter{\bar{\mb{z}}}{j}-\iter{\mb{z}}{j+1}\big)}{\mb{x}-\iter{\mb{x}}{j+1}}
	-
	\ts\frac{\sqrt{2\varepsilon_j}}{\sigma}\norm{\mb{x}-\iter{\mb{x}}{j+1}}{2}
	-
	\varepsilon_j \, . \label{eq:jour-14-fxr-supp:main_ineq_step_1}
\end{align}
Furthermore,
\begin{align}
	&\,\,\,\,\,\,\,\,
	\biprod{{-\mb{A}}'\big(\iter{\bar{\mb{z}}}{j}-\iter{\mb{z}}{j+1}\big)}{\mb{x}-\iter{\mb{x}}{j+1}} \nonumber \\
	&=
	\biprod{{-\mb{A}}'\big(\iter{\mb{z}}{j+1}-2\iter{\mb{z}}{j}+\iter{\mb{z}}{j-1}\big)}{\iter{\mb{x}}{j+1}-\mb{x}} \nonumber \\
	&=
	\biprod{{-\mb{A}}'\big(\iter{\mb{z}}{j+1}-\iter{\mb{z}}{j}\big)}{\iter{\mb{x}}{j+1}-\mb{x}}
	-
	\biprod{{-\mb{A}}'\big(\iter{\mb{z}}{j}-\iter{\mb{z}}{j-1}\big)}{\iter{\mb{x}}{j}-\mb{x}}
	-
	\biprod{{-\mb{A}}'\big(\iter{\mb{z}}{j}-\iter{\mb{z}}{j-1}\big)}{\iter{\mb{x}}{j+1}-\iter{\mb{x}}{j}} \nonumber \\
	&\geq
	\biprod{{-\mb{A}}'\big(\iter{\mb{z}}{j+1}-\iter{\mb{z}}{j}\big)}{\iter{\mb{x}}{j+1}-\mb{x}}
	-
	\biprod{{-\mb{A}}'\big(\iter{\mb{z}}{j}-\iter{\mb{z}}{j-1}\big)}{\iter{\mb{x}}{j}-\mb{x}}
	-
	\norm{\mb{A}}{2}\bnorm{\iter{\mb{z}}{j}-\iter{\mb{z}}{j-1}}{2}\bnorm{\iter{\mb{x}}{j+1}-\iter{\mb{x}}{j}}{2} \nonumber \\
	&\geq
	\biprod{{-\mb{A}}'\big(\iter{\mb{z}}{j+1}-\iter{\mb{z}}{j}\big)}{\iter{\mb{x}}{j+1}-\mb{x}}
	-
	\biprod{{-\mb{A}}'\big(\iter{\mb{z}}{j}-\iter{\mb{z}}{j-1}\big)}{\iter{\mb{x}}{j}-\mb{x}} \nonumber \\
	&\qquad\qquad\qquad\qquad\qquad\qquad\qquad\qquad\qquad\qquad\qquad\quad
	-
	\norm{\mb{A}}{2}
	\big(
	\ts\frac{\sqrt{\sigma/\tau}}{2}\bnorm{\iter{\mb{z}}{j}-\iter{\mb{z}}{j-1}}{2}^2
	+
	\ts\frac{1}{2\sqrt{\sigma/\tau}}\bnorm{\iter{\mb{x}}{j+1}-\iter{\mb{x}}{j}}{2}^2
	\big) \label{eq:jour-14-fxr-supp:young_ineq} \\
	&\geq
	\biprod{{-\mb{A}}'\big(\iter{\mb{z}}{j+1}-\iter{\mb{z}}{j}\big)}{\iter{\mb{x}}{j+1}-\mb{x}}
	-
	\biprod{{-\mb{A}}'\big(\iter{\mb{z}}{j}-\iter{\mb{z}}{j-1}\big)}{\iter{\mb{x}}{j}-\mb{x}} \nonumber \\
	&\qquad\qquad\qquad\qquad\qquad\qquad\qquad\qquad\qquad\qquad\qquad\qquad\quad
	-
	\sqrt{\sigma\tau}\norm{\mb{A}}{2}
	\left(\frac{\bnorm{\iter{\mb{z}}{j}-\iter{\mb{z}}{j-1}}{2}^2}{2\tau}+\frac{\bnorm{\iter{\mb{x}}{j+1}-\iter{\mb{x}}{j}}{2}^2}{2\sigma}\right)
	\, , \label{eq:report-13-tca:inner_prod_ineq}
\end{align}
where \eqref{eq:jour-14-fxr-supp:young_ineq} is due to Young's inequality. Plugging \eqref{eq:report-13-tca:inner_prod_ineq} into \eqref{eq:jour-14-fxr-supp:main_ineq_step_1}, it follows that for any $\left(\mb{z},\mb{x}\right)$,
\begin{align}
	&
	\frac{\norm{\iter{\mb{x}}{j}-\mb{x}}{2}^2}{2\sigma}+\frac{\norm{\iter{\mb{z}}{j}-\mb{z}}{2}^2}{2\tau}
	\geq
	\left(\bfx{\Omega}{\iter{\mb{z}}{j+1},\mb{x}}-\bfx{\Omega}{\mb{z},\iter{\mb{x}}{j+1}}\right)
	+
	\frac{\norm{\iter{\mb{x}}{j+1}-\mb{x}}{2}^2}{2\sigma}+\frac{\norm{\iter{\mb{z}}{j+1}-\mb{z}}{2}^2}{2\tau} \nonumber \\
	&\qquad\qquad
	+
	\left(1-\sqrt{\sigma\tau}\norm{\mb{A}}{2}\right)\frac{\norm{\iter{\mb{x}}{j+1}-\iter{\mb{x}}{j}}{2}^2}{2\sigma}
	+
	\frac{\norm{\iter{\mb{z}}{j+1}-\iter{\mb{z}}{j}}{2}^2}{2\tau}
	-
	\sqrt{\sigma\tau}\norm{\mb{A}}{2}\frac{\norm{\iter{\mb{z}}{j}-\iter{\mb{z}}{j-1}}{2}^2}{2\tau} \nonumber \\
	&\qquad\qquad
	+
	\biprod{{-\mb{A}}'\big(\iter{\mb{z}}{j+1}-\iter{\mb{z}}{j}\big)}{\iter{\mb{x}}{j+1}-\mb{x}}
	-
	\biprod{{-\mb{A}}'\big(\iter{\mb{z}}{j}-\iter{\mb{z}}{j-1}\big)}{\iter{\mb{x}}{j}-\mb{x}}
	-
	\ts\frac{\sqrt{2\varepsilon_j}}{\sigma}\norm{\mb{x}-\iter{\mb{x}}{j+1}}{2}
	-
	\varepsilon_j \, . \label{eq:jour-14-fxr-supp:main_ineq_step_2}
\end{align}
Suppose $\iter{\mb{z}}{-1}=\iter{\mb{z}}{0}$, i.e., $\iter{\bar{\mb{z}}}{0}=\iter{\mb{z}}{0}$. Summing up \eqref{eq:jour-14-fxr-supp:main_ineq_step_2} from $j=0,\ldots,k-1$ and using
\begin{equation} \label{eq:jour-14-fxr-supp:inner_prod_ineq2}
	\biprod{{-\mb{A}}'\big(\iter{\mb{z}}{k}-\iter{\mb{z}}{k-1}\big)}{\iter{\mb{x}}{k}-\mb{x}}
	\leq
	\frac{\norm{\iter{\mb{z}}{k}-\iter{\mb{z}}{k-1}}{2}^2}{2\tau}
	+
	\sigma\tau\norm{\mb{A}}{2}^2
	\frac{\norm{\iter{\mb{x}}{k}-\mb{x}}{2}^2}{2\sigma}
\end{equation}
as before, we have
\begin{align}
	&
	\sum_{j=1}^k\left(\bfx{\Omega}{\iter{\mb{z}}{j},\mb{x}}-\bfx{\Omega}{\mb{z},\iter{\mb{x}}{j}}\right)
	+
	\left(1-\sigma\tau\norm{\mb{A}}{2}^2\right)\frac{\norm{\iter{\mb{x}}{k}-\mb{x}}{2}^2}{2\sigma}
	+
	\frac{\norm{\iter{\mb{z}}{k}-\mb{z}}{2}^2}{2\tau} \nonumber \\
	&
	+
	\left(1-\sqrt{\sigma\tau}\norm{\mb{A}}{2}\right)\sum_{j=1}^k\frac{\norm{\iter{\mb{x}}{j}-\iter{\mb{x}}{j-1}}{2}^2}{2\sigma}
	+
	\left(1-\sqrt{\sigma\tau}\norm{\mb{A}}{2}\right)\sum_{j=1}^{k-1}\frac{\norm{\iter{\mb{z}}{j}-\iter{\mb{z}}{j-1}}{2}^2}{2\tau} \nonumber \\
	&\qquad\qquad\qquad\qquad\leq
	\frac{\norm{\iter{\mb{x}}{0}-\mb{x}}{2}^2}{2\sigma}
	+
	\frac{\norm{\iter{\mb{z}}{0}-\mb{z}}{2}^2}{2\tau}
	+
	\sum_{j=1}^k\varepsilon_{j-1}
	+
	\sum_{j=1}^k 2\sqrt{\ts\frac{\varepsilon_{j-1}}{\sigma}}\frac{\norm{\iter{\mb{x}}{j}-\mb{x}}{2}}{\sqrt{2\sigma}} \, . \label{eq:jour-14-fxr-supp:main_ineq_step_3}
\end{align}
Since $\sigma\tau\norm{\mb{A}}{2}^{2}<1$, we have $1-\sigma\tau\norm{\mb{A}}{2}^2>0$ and $1-\sqrt{\sigma\tau}\norm{\mb{A}}{2}>0$. If we choose $\left(\mb{z},\mb{x}\right)=\left(\hat{\mb{z}},\hat{\mb{x}}\right)$, the first term on the left-hand side of \eqref{eq:jour-14-fxr-supp:main_ineq_step_3} is the sum of $k$ non-negative primal-dual gaps, and all terms in \eqref{eq:jour-14-fxr-supp:main_ineq_step_3} are greater than or equal to zero. Let $D\teq1-\sigma\tau\norm{\mb{A}}{2}^2>0$. We have three inequalities:
\begin{equation} \label{eq:jour-14-fxr-supp:main_ineq_1}
	D\cdot\frac{\norm{\iter{\mb{x}}{k}-\hat{\mb{x}}}{2}^2}{2\sigma}
	\leq
	\frac{\norm{\iter{\mb{x}}{0}-\hat{\mb{x}}}{2}^2}{2\sigma}
	+
	\frac{\norm{\iter{\mb{z}}{0}-\hat{\mb{z}}}{2}^2}{2\tau}
	+
	\sum_{j=1}^k\varepsilon_{j-1}
	+
	\sum_{j=1}^k 2\sqrt{\ts\frac{\varepsilon_{j-1}}{\sigma}}\frac{\norm{\iter{\mb{x}}{j}-\hat{\mb{x}}}{2}}{\sqrt{2\sigma}} \, ,
\end{equation}
\begin{equation} \label{eq:jour-14-fxr-supp:main_ineq_2}
	D\cdot
	\left(\frac{\norm{\iter{\mb{x}}{k}-\hat{\mb{x}}}{2}^2}{2\sigma}+\frac{\norm{\iter{\mb{z}}{k}-\hat{\mb{z}}}{2}^2}{2\tau}\right)
	\leq
	\frac{\norm{\iter{\mb{x}}{0}-\hat{\mb{x}}}{2}^2}{2\sigma}
	+
	\frac{\norm{\iter{\mb{z}}{0}-\hat{\mb{z}}}{2}^2}{2\tau}
	+
	\sum_{j=1}^k\varepsilon_{j-1}
	+
	\sum_{j=1}^k 2\sqrt{\ts\frac{\varepsilon_{j-1}}{\sigma}}\frac{\norm{\iter{\mb{x}}{j}-\hat{\mb{x}}}{2}}{\sqrt{2\sigma}} \, ,
\end{equation}
and
\begin{equation} \label{eq:jour-14-fxr-supp:main_ineq_3}
	\sum_{j=1}^k\left(\bfx{\Omega}{\iter{\mb{z}}{j},\hat{\mb{x}}}-\bfx{\Omega}{\hat{\mb{z}},\iter{\mb{x}}{j}}\right)
	\leq
	\frac{\norm{\iter{\mb{x}}{0}-\hat{\mb{x}}}{2}^2}{2\sigma}
	+
	\frac{\norm{\iter{\mb{z}}{0}-\hat{\mb{z}}}{2}^2}{2\tau}
	+
	\sum_{j=1}^k\varepsilon_{j-1}
	+
	\sum_{j=1}^k 2\sqrt{\ts\frac{\varepsilon_{j-1}}{\sigma}}\frac{\norm{\iter{\mb{x}}{j}-\hat{\mb{x}}}{2}}{\sqrt{2\sigma}} \, .
\end{equation}
All these inequality has a common right-hand-side. To continue the proof, we have to bound $\norm{\iter{\mb{x}}{j}-\hat{\mb{x}}}{2}/\sqrt{2\sigma}$ first. Dividing $D$ from both sides of \eqref{eq:jour-14-fxr-supp:main_ineq_1}, we have
\begin{equation} \label{eq:jour-14-fxr-supp:main_ineq_4}
	\left(\frac{\norm{\iter{\mb{x}}{k}-\hat{\mb{x}}}{2}}{\sqrt{2\sigma}}\right)^2
	\leq
	\left(
	\frac{1}{D}\frac{\norm{\iter{\mb{x}}{0}-\hat{\mb{x}}}{2}^2}{2\sigma}
	+
	\frac{1}{D}\frac{\norm{\iter{\mb{z}}{0}-\hat{\mb{z}}}{2}^2}{2\tau}
	+
	\sum_{j=1}^k\frac{\varepsilon_{j-1}}{D}
	\right)
	+
	\sum_{j=1}^k 2\left(\frac{1}{D}\sqrt{\ts\frac{\varepsilon_{j-1}}{\sigma}}\right)\frac{\norm{\iter{\mb{x}}{j}-\hat{\mb{x}}}{2}}{\sqrt{2\sigma}} \, .
\end{equation}
Let
\begin{equation} \label{eq:jour-14-fxr-supp:def_Sk}
	S_k
	\teq
	\frac{1}{D}\frac{\norm{\iter{\mb{x}}{0}-\hat{\mb{x}}}{2}^2}{2\sigma}
	+
	\frac{1}{D}\frac{\norm{\iter{\mb{z}}{0}-\hat{\mb{z}}}{2}^2}{2\tau}
	+
	\sum_{j=1}^k\frac{\varepsilon_{j-1}}{D} \, ,
\end{equation}
\begin{equation} \label{eq:jour-14-fxr-supp:def_lambdaj}
	\lambda_j
	\teq
	2\left(\frac{1}{D}\sqrt{\ts\frac{\varepsilon_{j-1}}{\sigma}}\right) \, ,
\end{equation}
and
\begin{equation} \label{eq:jour-14-fxr-supp:def_uj}
	u_j
	\teq
	\frac{\norm{\iter{\mb{x}}{j}-\hat{\mb{x}}}{2}}{\sqrt{2\sigma}} \, .
\end{equation}
We have $u_k^2\leq S_k+\sum_{j=1}^k\lambda_j u_j$ from \eqref{eq:jour-14-fxr-supp:main_ineq_4} with $\left\{S_k\right\}_{k=0}^{\infty}$ an increasing sequence, $S_0\geq u_0^2$ (note that $0<D<1$ because $0<\sigma\tau\norm{\mb{A}}{2}^2<1$), and $\lambda_j\geq 0$ for all $j$. According to \cite[Lemma 1]{schmidt:11:cro}, it follows that
\begin{equation} \label{eq:jour-14-fxr-supp:bound_uk}
	\frac{\norm{\iter{\mb{x}}{k}-\hat{\mb{x}}}{2}}{\sqrt{2\sigma}}
	\leq
	\widetilde{A}_k
	+
	\left(
	\frac{1}{D}\frac{\norm{\iter{\mb{x}}{0}-\hat{\mb{x}}}{2}^2}{2\sigma}
	+
	\frac{1}{D}\frac{\norm{\iter{\mb{z}}{0}-\hat{\mb{z}}}{2}^2}{2\tau}
	+
	\widetilde{B}_k
	+
	\widetilde{A}_k^2
	\right)^{1/2} \, ,
\end{equation}
where
\begin{equation} \label{eq:jour-14-fxr-supp:def_tildeAk}
	\widetilde{A}_k
	\teq
	\sum_{j=1}^k\frac{1}{D}\sqrt{\ts\frac{\varepsilon_{j-1}}{\sigma}} \, ,
\end{equation}
and
\begin{equation} \label{eq:jour-14-fxr-supp:def_tildeBk}
	\widetilde{B}_k
	\teq
	\sum_{j=1}^k\frac{\varepsilon_{j-1}}{D} \, .
\end{equation}
Since $\widetilde{A}_j$ and $\widetilde{B}_j$ are increasing sequences of $j$, for $j\leq k$, we have
\begin{align}
	\frac{\norm{\iter{\mb{x}}{j}-\hat{\mb{x}}}{2}}{\sqrt{2\sigma}}
	&\leq
	\widetilde{A}_j
	+
	\left(
	\frac{1}{D}\frac{\norm{\iter{\mb{x}}{0}-\hat{\mb{x}}}{2}^2}{2\sigma}
	+
	\frac{1}{D}\frac{\norm{\iter{\mb{z}}{0}-\hat{\mb{z}}}{2}^2}{2\tau}
	+
	\widetilde{B}_j
	+
	\widetilde{A}_j^2
	\right)^{1/2} \nonumber \\
	&\leq
	\widetilde{A}_k
	+
	\left(
	\frac{1}{D}\frac{\norm{\iter{\mb{x}}{0}-\hat{\mb{x}}}{2}^2}{2\sigma}
	+
	\frac{1}{D}\frac{\norm{\iter{\mb{z}}{0}-\hat{\mb{z}}}{2}^2}{2\tau}
	+
	\widetilde{B}_k
	+
	\widetilde{A}_k^2
	\right)^{1/2} \nonumber \\
	&\leq
	\widetilde{A}_k
	+
	\left(
	\frac{1}{\sqrt{D}}\frac{\norm{\iter{\mb{x}}{0}-\hat{\mb{x}}}{2}}{\sqrt{2\sigma}}
	+
	\frac{1}{\sqrt{D}}\frac{\norm{\iter{\mb{z}}{0}-\hat{\mb{z}}}{2}}{\sqrt{2\tau}}
	+
	\sqrt{\widetilde{B}_k}
	+
	\widetilde{A}_k
	\right) \, . \label{eq:jour-14-fxr-supp:bound_uj}
\end{align}
Now, we can bound the right-hand-side of \eqref{eq:jour-14-fxr-supp:main_ineq_1}, \eqref{eq:jour-14-fxr-supp:main_ineq_2}, and \eqref{eq:jour-14-fxr-supp:main_ineq_3} as
\begin{align}
	&\,\,\,\,\,\,\,\,
	\frac{\norm{\iter{\mb{x}}{0}-\hat{\mb{x}}}{2}^2}{2\sigma}
	+
	\frac{\norm{\iter{\mb{z}}{0}-\hat{\mb{z}}}{2}^2}{2\tau}
	+
	\sum_{j=1}^k\varepsilon_{j-1}
	+
	\sum_{j=1}^k 2\sqrt{\ts\frac{\varepsilon_{j-1}}{\sigma}}\frac{\norm{\iter{\mb{x}}{j}-\hat{\mb{x}}}{2}}{\sqrt{2\sigma}} \nonumber \\
	&\leq
	\frac{\norm{\iter{\mb{x}}{0}-\hat{\mb{x}}}{2}^2}{2\sigma}
	+
	\frac{\norm{\iter{\mb{z}}{0}-\hat{\mb{z}}}{2}^2}{2\tau}
	+
	\sum_{j=1}^k\varepsilon_{j-1}
	+
	\sum_{j=1}^k2\sqrt{\ts\frac{\varepsilon_{j-1}}{\sigma}}
	\left(
	2\widetilde{A}_k
	+
	\frac{1}{\sqrt{D}}\frac{\norm{\iter{\mb{x}}{0}-\hat{\mb{x}}}{2}}{\sqrt{2\sigma}}
	+
	\frac{1}{\sqrt{D}}\frac{\norm{\iter{\mb{z}}{0}-\hat{\mb{z}}}{2}}{\sqrt{2\tau}}
	+
	\sqrt{\widetilde{B}_k}
	\right) \nonumber \\
	&=
	\frac{\norm{\iter{\mb{x}}{0}-\hat{\mb{x}}}{2}^2}{2\sigma}
	+
	\frac{\norm{\iter{\mb{z}}{0}-\hat{\mb{z}}}{2}^2}{2\tau}
	+
	\widetilde{B}_kD
	+
	2\widetilde{A}_kD
	\left(
	2\widetilde{A}_k
	+
	\frac{1}{\sqrt{D}}\frac{\norm{\iter{\mb{x}}{0}-\hat{\mb{x}}}{2}}{\sqrt{2\sigma}}
	+
	\frac{1}{\sqrt{D}}\frac{\norm{\iter{\mb{z}}{0}-\hat{\mb{z}}}{2}}{\sqrt{2\tau}}
	+
	\sqrt{\widetilde{B}_k}
	\right) \nonumber \\
	&\leq
	\left(
	\frac{\norm{\iter{\mb{x}}{0}-\hat{\mb{x}}}{2}}{\sqrt{2\sigma}}
	+
	\frac{\norm{\iter{\mb{z}}{0}-\hat{\mb{z}}}{2}}{\sqrt{2\tau}}
	+
	2\widetilde{A}_k\sqrt{D}
	+
	\sqrt{\widetilde{B}_kD}
	\right)^2 \nonumber \\
	&=
	\left(
	\frac{\norm{\iter{\mb{x}}{0}-\hat{\mb{x}}}{2}}{\sqrt{2\sigma}}
	+
	\frac{\norm{\iter{\mb{z}}{0}-\hat{\mb{z}}}{2}}{\sqrt{2\tau}}
	+
	2A_k
	+
	\sqrt{B_k}
	\right)^2 \label{eq:jour-14-fxr-supp:rhs_upper_bd_1} \\
	&\leq
	\left(
	\frac{\norm{\iter{\mb{x}}{0}-\hat{\mb{x}}}{2}}{\sqrt{2\sigma}}
	+
	\frac{\norm{\iter{\mb{z}}{0}-\hat{\mb{z}}}{2}}{\sqrt{2\tau}}
	+
	2A_{\infty}
	+
	\sqrt{B_{\infty}}
	\right)^2 \label{eq:jour-14-fxr-supp:rhs_upper_bd_2}
\end{align}
if $\left\{\sqrt{\varepsilon_k}\right\}_{k=0}^{\infty}$ is absolutely summable (and therefore, $\left\{\varepsilon_k\right\}_{k=0}^{\infty}$ is also absolutely summable), where
\begin{equation} 
	A_k
	\teq
	\widetilde{A}_k\sqrt{D}
	=
	\sum_{j=1}^k\sqrt{\ts\frac{\varepsilon_{j-1}}{\left(1-\sigma\tau\norm{\mb{A}}{2}^2\right)\sigma}} \, ,
\end{equation}
and
\begin{equation} 
	B_k
	\teq
	\widetilde{B}_kD
	=
	\sum_{j=1}^k\varepsilon_{j-1} \, .
\end{equation}
Hence, from \eqref{eq:jour-14-fxr-supp:main_ineq_2}, we have
\begin{equation} \label{eq:jour-14-fxr-supp:bdd_seq}
	\frac{\norm{\iter{\mb{x}}{k}-\hat{\mb{x}}}{2}^2}{2\sigma}+\frac{\norm{\iter{\mb{z}}{k}-\hat{\mb{z}}}{2}^2}{2\tau}
	\leq
	\frac{1}{D}
	\left(
	\frac{\norm{\iter{\mb{x}}{0}-\hat{\mb{x}}}{2}}{\sqrt{2\sigma}}
	+
	\frac{\norm{\iter{\mb{z}}{0}-\hat{\mb{z}}}{2}}{\sqrt{2\tau}}
	+
	2A_{\infty}
	+
	\sqrt{B_{\infty}}
	\right)^2
	<
	\infty \, .
\end{equation}
This implies that the sequence of updates $\left\{\left(\iter{\mb{z}}{k},\iter{\mb{x}}{k}\right)\right\}_{k=0}^{\infty}$ generated by the inexact CPPDA \eqref{eq:jour-14-fxr-supp:inexact_cppd_type_2} is a bounded sequence. Let
\begin{equation} 
	C
	\teq
	\frac{\norm{\iter{\mb{x}}{0}-\hat{\mb{x}}}{2}}{\sqrt{2\sigma}}
	+
	\frac{\norm{\iter{\mb{z}}{0}-\hat{\mb{z}}}{2}}{\sqrt{2\tau}} \, .
\end{equation}
From \eqref{eq:jour-14-fxr-supp:main_ineq_3} and the convexity of $h$ and $g^*$, we have
\begin{align}
	\bfx{\Omega}{\mb{z}_k,\hat{\mb{x}}}-\bfx{\Omega}{\hat{\mb{z}},\mb{x}_k}
	&\leq
	\frac{1}{k}\sum_{j=1}^k\left(\bfx{\Omega}{\iter{\mb{z}}{j},\hat{\mb{x}}}-\bfx{\Omega}{\hat{\mb{z}},\iter{\mb{x}}{j}}\right) \nonumber \\
	&\leq
	\frac{\left(C+2A_k+\sqrt{B_k}\right)^2}{k} \label{eq:jour-14-fxr-supp:primal_dual_gap_bd_1} \\
	&\leq
	\frac{\left(C+2A_{\infty}+\sqrt{B_{\infty}}\right)^2}{k} \label{eq:jour-14-fxr-supp:primal_dual_gap_bd_2} \, ,
\end{align}
where $\mb{z}_k\teq\frac{1}{k}\sum_{j=1}^k\iter{\mb{z}}{j}$, and $\mb{x}_k\teq\frac{1}{k}\sum_{j=1}^k\iter{\mb{x}}{j}$. That is, the primal-dual gap of $\left(\mb{z}_k,\mb{x}_k\right)$ converges to zero with rate $\fx{O}{1/k}$. Following the procedure in \cite[Section 3.1]{chambolle:11:afo}, we can further show that the sequence of updates $\left\{\left(\iter{\mb{z}}{k},\iter{\mb{x}}{k}\right)\right\}_{k=0}^{\infty}$ generated by the inexact CPPDA \eqref{eq:jour-14-fxr-supp:inexact_cppd_type_2} converges to a saddle-point of \eqref{eq:jour-14-fxr-supp:eq_comp_conv_opt_prob_minimax} if the dimension of $\mb{x}$ and $\mb{z}$ is finite.
\end{proof}

\section{Additional experimental results} \label{sec:jour-14-fxr:additional_expt}
\subsection{Shoulder scan with the Barzilai-Borwein acceleration} \label{subsec:jour-14-fxr:shoulder_bb}
\begin{figure}
	\centering
	\includegraphics[width=0.45\textwidth]{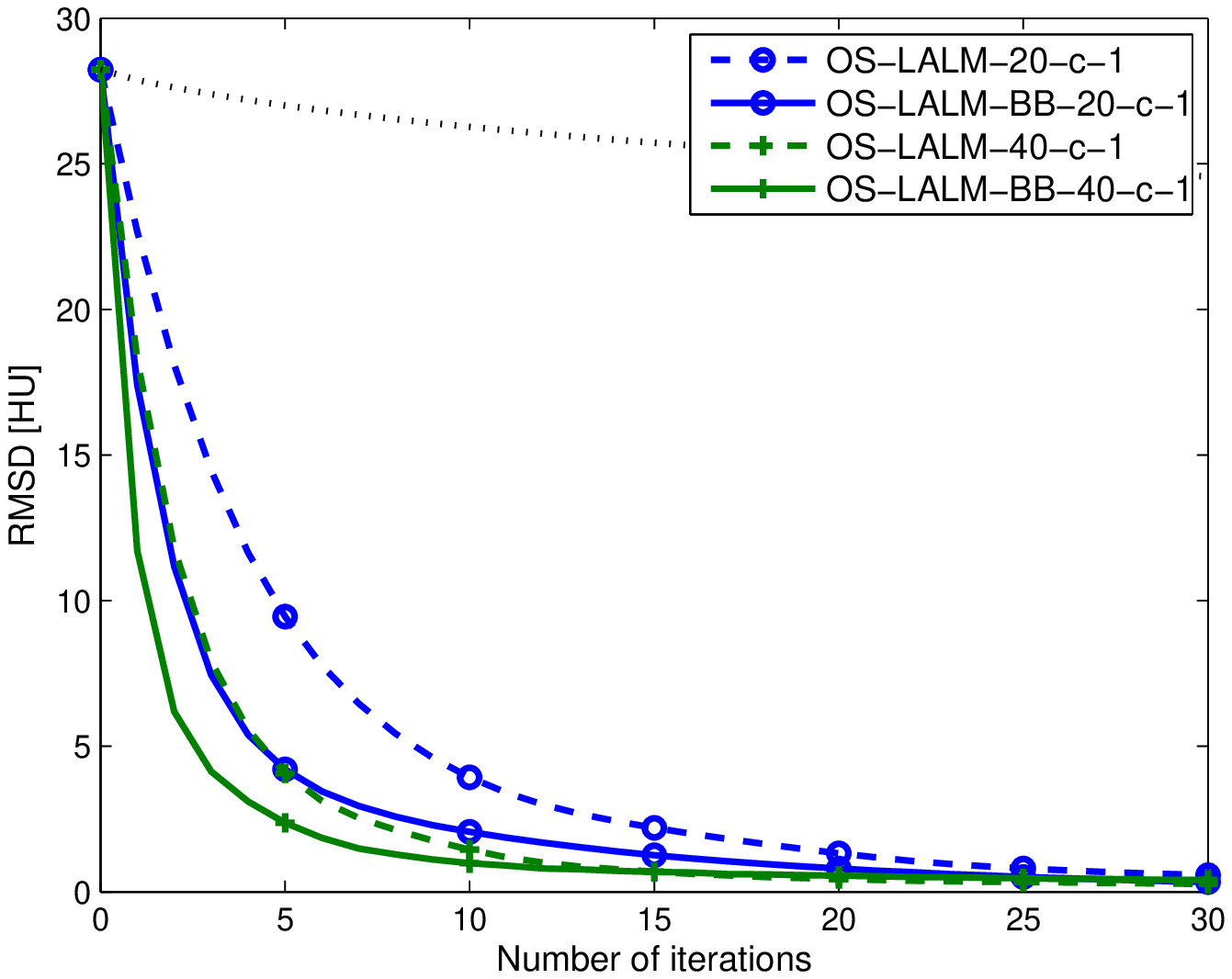}
	\caption{Shoulder scan: RMS differences between the reconstructed image $\iter{\mb{x}}{k}$ and the reference reconstruction $\mb{x}^{\star}$ as a function of iteration using the proposed algorithm without and with the Barzilai-Borwein acceleration. The dotted line shows the RMS differences using the standard OS algorithm with one subset as the baseline convergence rate.}
	\label{fig:jour-14-fxr-supp:shoulder_rmsd_bb}
\end{figure}
In this experiment, we demonstrated accelerating the proposed algorithm using the Barzilai-Borwein (spectral) method \cite{barzilai:88:tps} that mimics the Hessian $\mb{A}'\mb{WA}$ by $\mb{H}_k\teq\alpha_k\mb{L}_{\text{diag}}$. The scaling factor $\alpha_k$ is solved by fitting the secant equation:
\begin{equation} \label{eq:jour-14-fxr-supp:secant_eq}
	\mb{y}_k
	\approx
	\mb{H}_k\mb{s}_k
\end{equation}
in the weighted least-squares sense, i.e.,
\begin{equation} \label{eq:jour-14-fxr-supp:scale_factor_alpha_k}
	\alpha_k
	\in
	\nbargmin{\alpha\leq 1}{\ts\frac{1}{2}\norm{\mb{y}_k-\alpha\mb{L}_{\text{diag}}\mb{s}_k}{\mb{P}}^2}
\end{equation}
for some positive definite $\mb{P}$, where
\begin{equation}
	\mb{y}_k\teq\bfx{\nabla\ell}{\iter{\mb{x}}{k}}-\bfx{\nabla\ell}{\iter{\mb{x}}{k-1}}
\end{equation}
and
\begin{equation}
	\mb{s}_k\teq\iter{\mb{x}}{k}-\iter{\mb{x}}{k-1} \, .
\end{equation}
We choose $\mb{P}$ to be $\mb{L}_{\text{diag}}^{-1}$ since $\mb{L}_{\text{diag}}^{-1}$ is proportional to the step sizes of the voxels. By choose $\mb{P}=\mb{L}_{\text{diag}}^{-1}$, we are fitting the secant equation with more weight for voxels with larger step sizes. Note that applying the Barzilai-Borwein acceleration changes $\mb{H}_k$ every iteration, and the majorization condition does not necessarily hold. Hence, the convergence theorems developed in Section~\ref{sec:jour-14-fxr-supp:conv_ana_inexact_lalm} are not applicable. However, ordered-subsets (OS) based algorithms typically lack convergence proofs anyway, and we find that this acceleration works well in practice. Figure~\ref{fig:jour-14-fxr-supp:shoulder_rmsd_bb} shows the RMS differences between the reconstructed image $\iter{\mb{x}}{k}$ and the reference reconstruction $\mb{x}^{\star}$ of the shoulder scan dataset as a function of iteration using the proposed algorithm without and with the Barzilai-Borwein acceleration. As can be seen in Figure~\ref{fig:jour-14-fxr-supp:shoulder_rmsd_bb}, the proposed algorithm with both $M=20$ and $M=40$ shows roughly $2$-times acceleration in early iterations using the Barzilai-Borwein acceleration.

\subsection{Truncated abdomen scan} \label{subsec:jour-14-fxr:trunc}
In this experiment, we reconstructed a $600\times 600\times 239$ image from an abdomen region helical CT scan with transaxial truncation, where the sinogram has size $888\times 64\times 3516$ and pitch $1.0$. The maximum number of subsets suggested in \cite{nien:14:fxr} is about $20$. Figure~\ref{fig:jour-14-fxr:truncation_ini_ref_pro} shows the cropped images from the central transaxial plane of the initial FBP image, the reference reconstruction, and the reconstructed image using the proposed algorithm (\mbox{OS-LALM-$20$-c-$1$}) at the $30$th iteration. This experiment demonstrates how different OS-based algorithms behave when the number of subsets exceeds the suggested maximum number of subsets. Figure~\ref{fig:jour-14-fxr:truncation_diff_image} shows the difference images for different OS-based algorithms with $10$, $20$, and $40$ subsets. As can be seen in Figure~\ref{fig:jour-14-fxr:truncation_diff_image}, the proposed algorithm works best for $M=20$; when $M$ is larger ($M=40$), ripples and light OS artifacts appear. However, it is still much better than the standard OS+momentum algorithm \cite{kim:13:axr}. In fact, the OS artifacts in the reconstructed image using the standard OS+momentum algorithm with $40$ subsets are visible with the naked eye in the display window from $800$ to $1200$ HU. The convergence rate curves in Figure~\ref{fig:jour-14-fxr:truncation_rmsd} support our observation. In sum, the proposed algorithm exhibits fast convergence rate and excellent gradient error tolerance even in the case with truncation.

\begin{figure}
	\centering
	\includegraphics[width=\textwidth]{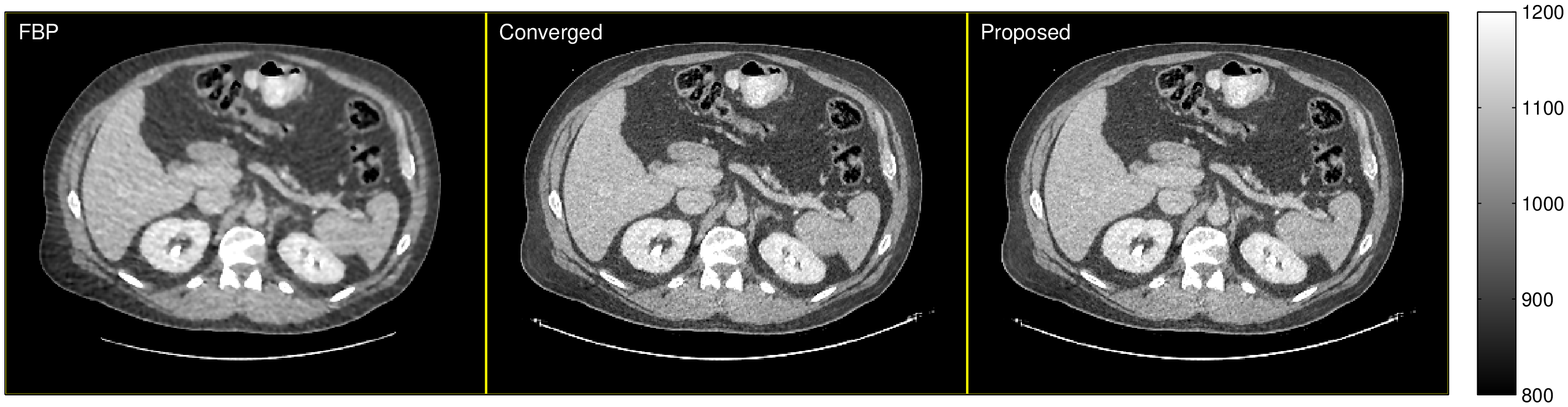}
	\caption{Truncated abdomen scan: cropped images (displayed from $800$ to $1200$ HU) from the central transaxial plane of the initial FBP image $\iter{\mb{x}}{0}$ (left), the reference reconstruction $\mb{x}^{\star}$ (center), and the reconstructed image using the proposed algorithm (\mbox{OS-LALM-$20$-c-$1$}) at the $30$th iteration $\iter{\mb{x}}{30}$ (right).}
	\label{fig:jour-14-fxr:truncation_ini_ref_pro}
\end{figure}

\begin{figure}
	\centering
	\includegraphics[width=\textwidth]{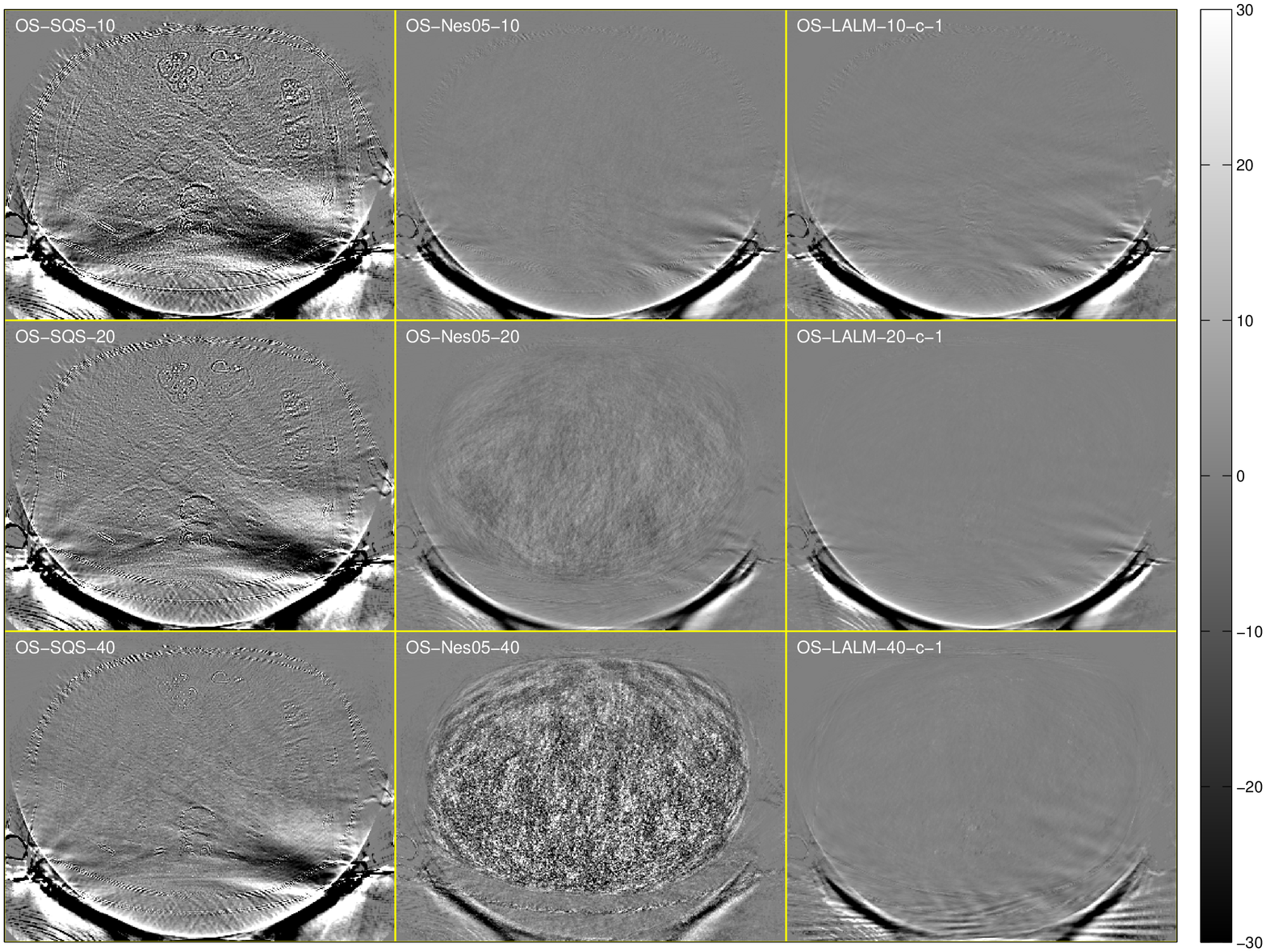}
	\caption{Truncated abdomen scan: cropped difference images (displayed from ${-30}$ to $30$ HU) from the central transaxial plane of $\iter{\mb{x}}{30}-\mb{x}^{\star}$ using OS-based algorithms.}
	\label{fig:jour-14-fxr:truncation_diff_image}
\end{figure}

\begin{figure}
	\centering
	\includegraphics[width=0.45\textwidth]{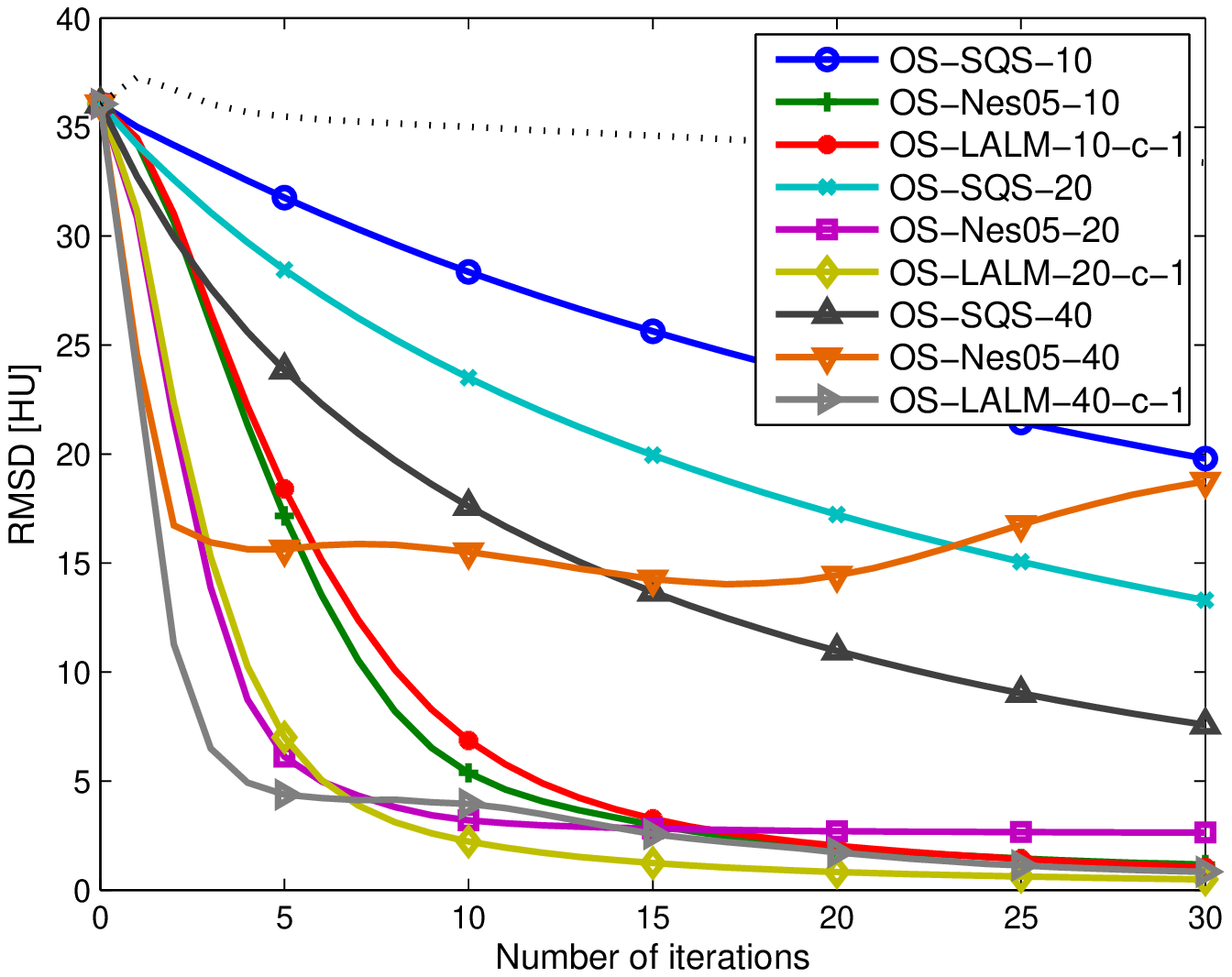}
	\caption{Truncated abdomen scan: RMS differences between the reconstructed image $\iter{\mb{x}}{k}$ and the reference reconstruction $\mb{x}^{\star}$ as a function of iteration using OS-based algorithms with $10$, $20$, and $40$ subsets. The dotted line shows the RMS differences using the standard OS algorithm with one subset as the baseline convergence rate.}
	\label{fig:jour-14-fxr:truncation_rmsd}
\end{figure}

\bibliographystyle{ieeetr}
\bibliography{../master}

%% file: def.tex
\newcommand{\mb}[1]{\mathbf{#1}}
\newcommand{\msb}[1]{\boldsymbol{#1}}

\newcommand{\mbi}[1]{\boldsymbol{\mathit{#1}}}

\newcommand{\ts}{\textstyle}

\newcommand{\sscs}{\scriptscriptstyle}

\newcommand{\teq}{\triangleq}
\newcommand{\fx}[2]{#1\!\left(#2\right)}
\newcommand{\bfx}[2]{#1\big(#2\big)}

\newcommand{\sfx}[2]{#1(#2)}
\newcommand{\varmin}[2]{\underset{#1}{\text{min}}\left\{#2\right\}}
\newcommand{\bvarmin}[2]{\underset{#1}{\text{min}}\,\big\{#2\!\big\}}

\newcommand{\nbvarmin}[2]{\underset{#1}{\text{min}}\,#2}
\newcommand{\argmin}[2]{\text{arg}\,\varmin{#1}{#2}}
\newcommand{\bargmin}[2]{\text{arg}\,\bvarmin{#1}{#2}}

\newcommand{\nbargmin}[2]{\text{arg}\,\nbvarmin{#1}{#2}}
\newcommand{\varmax}[2]{\underset{#1}{\text{max}}\left\{#2\right\}}

\newcommand{\nbvarmax}[2]{\underset{#1}{\text{max}}\,#2}
\newcommand{\argmax}[2]{\text{arg}\,\varmax{#1}{#2}}

\newcommand{\nbargminmax}[3]{\text{arg}\,\nbvarmin{#1}{}\nbvarmax{#2}{#3}}
\newcommand{\norm}[2]{\left\|#1\right\|_{#2}}

\newcommand{\Bnorm}[2]{\Big\|#1\Big\|_{#2}}

\newcommand{\iter}[2]{#1^{\left(#2\right)}}

\newcommand{\LA}{\mathcal{L}_{\text{A}}}
\newcommand{\R}{\mathsf{R}}
\newcommand{\prox}[2]{\text{prox}_{#1}\!\left(#2\right)}
\newcommand{\bprox}[2]{\text{prox}_{#1}\big(#2\big)}

\newcommand{\diag}[1]{\text{diag}\!\left\{#1\right\}}

\newcommand{\iprod}[2]{\left\langle#1,#2\right\rangle}
\newcommand{\biprod}[2]{\big\langle#1,#2\big\rangle}